\documentclass[11pt,a4paper]{article}
\usepackage[utf8]{inputenc}
\usepackage{amsmath,amssymb,latexsym,color,enumerate,amsthm,mathtools}
\usepackage{indentfirst}
\usepackage[mathscr]{eucal}
\usepackage[colorlinks,linkcolor=blue,anchorcolor=green,citecolor=magenta]{hyperref}
\usepackage{setspace}
\usepackage{authblk}
\allowdisplaybreaks

\newtheorem{thm}{Theorem}[section]
\newtheorem{prop}[thm]{Proposition}
\newtheorem{lemma}[thm]{Lemma}
\newtheorem{cor}[thm]{Corollary}
\newtheorem{conj}[thm]{Conjecture}
\newtheorem{obs}[thm]{Observation}
\newtheorem{rema}{Remark}
\newtheorem{defin}{Definition}[section]
\newenvironment{definition}{\begin{defin}\rm}{\end{defin}}

\newcommand{\Fq}{\mathbb F_q}
\newcommand{\calA}{\mathcal A}
\newcommand{\calB}{\mathcal B}

\newcommand{\calF}{\mathcal F}
\newcommand{\calG}{\mathcal G}
\newcommand{\calH}{\mathcal H}
\newcommand{\calL}{\mathcal L}
\newcommand{\calS}{\mathcal S}
\newcommand{\calT}{\mathcal T}
\newcommand{\qbinom}[2]{\genfrac{[}{]}{0pt}{}{#1}{#2}_{q}}
\newcommand{\vbinom}[2]{\genfrac{[}{]}{0pt}{}{#1}{#2}}

\begin{document}
\renewcommand{\baselinestretch}{1.3}

\title{Projective Ore-Degree Conditions for Intersection Theorems
in Vector Spaces}

\author[1]{Mengyu Cao\thanks{E-mail: \texttt{myucao@ruc.edu.cn}. M. Cao is supported by the National Natural Science Foundation of China (12301431) and Beijing Natural Science Foundation (1262010).}}
\author[2]{Mei Lu\thanks{E-mail: \texttt{lumei@tsinghua.edu.cn}. M. Lu is supported by the National Natural Science Foundation of China (Grant 12571372) and Beijing Natural Science Foundation (Grant 1262010).}}
\author[2]{Xuyang Yan\thanks{E-mail: \texttt{yanxuyan24@mails.tsinghua.edu.cn}}}
\author[2]{Haixiang Zhang\thanks{Corresponding author. E-mail: \texttt{zhang-hx22@mails.tsinghua.edu.cn.}}}

\affil[1]{\small Institute for Mathematical Sciences, Renmin University of China, Beijing 100086, China}
\affil[2]{\small Department of Mathematical Sciences, Tsinghua University, Beijing 100084, China}
\date{}
\maketitle

\begin{abstract}
Let $V$ be an $n$-dimensional vector space over the finite field $\Fq$, and
let $\calF\subsetneq\vbinom{V}{k}$.  The \emph{projective Ore-degree} of
$\calF$ is the minimum, over all $k$-subspaces $S\notin\calF$, of the sum
of the $\calF$-degrees of the projective points contained in $S$.  We prove
sharp projective Ore analogues of the vector-space
Erd\H{o}s--Ko--Rado and Hilton--Milner theorems.  The Ore--Erd\H{o}s--Ko--Rado
theorem holds for $n\ge2k+1$, with equality only for a full point-star.
For nontrivial intersecting families, we determine the sharp
Ore--Hilton--Milner threshold, together with the complete equality
classification, when $q\ge3$ and $n\ge2k+1$, or when $q\ge2$ and
$n\ge2k+2$.  We further determine a sharp projective Ore-degree threshold
forcing a direct-sum matching of size $s$ when $s\ge3$ and
$n\ge(2s-1)k-s+4$, and derive a multicolour Ramsey consequence.


\end{abstract}

\section{Introduction}\label{sec:intro}

\subsection{Background and motivation}

Intersection theorems form a central part of extremal set theory and finite
geometry.  The Erd\H{o}s--Ko--Rado theorem determines the maximum size of an
intersecting family of uniform subsets, while the Hilton--Milner theorem
determines the maximum under the additional requirement that the family have
no common element \cite{ErdosKoRado,HiltonMilner}.  Their vector-space
analogues replace subsets by subspaces and binomial coefficients by Gaussian
coefficients.  Foundational results of Hsieh and Frankl--Wilson established
the vector-space Erd\H{o}s--Ko--Rado theorem, and subsequent work determined
the vector-space Hilton--Milner extremum and the large nontrivial
$t$-intersecting families
\cite{Hsieh,FranklWilson,BlokhuisEtAl,WangXuZhang,CaoLvWangZhou}.
Recently, Ihringer and Kupavskii developed subspace spreadness and a
peeling--simplification procedure which gives a low-dimensional governing
structure for large $t$-intersecting families already in the natural range
$n\ge2k+1$ \cite{IhringerKupavskii}.  Their result provides an important
framework for studying the critical range of sharp intersection and
stability problems.
For general background on Erd\H{o}s--Ko--Rado theory and its algebraic
formulations, see \cite{GodsilMeagher}.

Classical intersection theorems control the total size of a family.  Degree
conditions instead measure how the family is distributed among the local
objects of the host space.  In uniform set systems, minimum-degree versions
of the Erd\H{o}s--Ko--Rado and Hilton--Milner theorems were developed in
\cite{HuangZhao,FranklHanHaoYi}.  In finite vector spaces, Shan and Zhou
proved a sharp $d$-degree Erd\H{o}s--Ko--Rado theorem
\cite{ShanZhouDegree}.  Ore-type hypotheses provide a different local-to-global
statistic: rather than controlling every individual degree, they control the
sum of degrees over each nonedge.  This point of view originates in Ore's
degree-sum criterion for Hamilton cycles in graphs \cite{Ore}.  Balogh,
Palmer and Raeisi recently introduced the corresponding parameter for
uniform hypergraphs and proved Ore-degree versions of the
Erd\H{o}s--Ko--Rado and Hilton--Milner theorems, as well as a matching
theorem \cite{BaloghPalmerRaeisi}.

The purpose of the present paper is to develop this Ore-degree programme on a
Grassmannian.  The passage from sets to subspaces involves two genuine
geometric choices.  First, the local objects should be projective points,
rather than nonzero vectors, so that degrees are invariant under scalar
multiplication and each $k$-subspace contains the same intrinsic collection of
local objects.  Second, for three or more subspaces, pairwise trivial
intersection does not imply rank additivity.  The appropriate analogue of a
set matching is therefore a \emph{direct-sum matching}.  This is also the
notion used in the recent vector-space Erd\H{o}s matching problem of Feng,
Shangguan, Yang and Zhang \cite{FengShangguanYangZhang}.

Our main contribution is an exact incidence inequality that converts a lower
bound on projective Ore-degree into a lower bound on the size of the family.
Combining it with the large-family classification and covering-number
analysis of Cao, Lv, Wang and Zhou gives the sharp Ore--Hilton--Milner
threshold when $q\ge3$ and $n\ge2k+1$, or when $q\ge2$ and $n\ge2k+2$.
For $k\ge4$, the two boundary layers follow from sharper comparisons with
the covering-number bounds underlying their classification; at the single
exceptional comparison $(q,k,n)=(3,4,9)$, an additional averaging argument
over subspaces disjoint from a fixed member closes the gap.  The case $k=3$
follows from De Boeck's classification of large intersecting families of
planes \cite{DeBoeck,IhringerKupavskii}; see
Appendix~\ref{app:k3-classification} for the explicit classification.  We also determine
every equality family: the outer part of the canonical construction is rigid,
while its members contained in the distinguished $(k+1)$-space may be chosen
freely subject only to nontriviality; for $k=3$ there are two further families
arising from the exceptional three-dimensional construction.
The only case left at the natural dimension boundary is
$q=2$ and $n=2k+1$.  The structure theorem of Ihringer and Kupavskii strongly
suggests that the same conclusion should hold there, but its present
remainder estimate does not by itself give the exact Ore threshold.  We record
this conjectural binary extension, together with the optimality of the
dimension range, in Remark~\ref{rema:hm-binary-boundary}.  We also obtain a
sharp large-dimension theorem for direct-sum matchings and a multicolour
Ramsey consequence.

\subsection{Definitions and intersection results}

Throughout the paper, $q$ is a prime power, $V$ is an $n$-dimensional vector
space over $\Fq$, and
\[
 \vbinom{V}{k}=\{F\le V:\dim F=k\}
\]
is the Grassmannian of $k$-subspaces of $V$.  Its cardinality is the Gaussian
coefficient
\[
 \qbinom{n}{k}=\prod_{i=0}^{k-1}\frac{q^{n-i}-1}{q^{k-i}-1}.
\]
We use the convention $\qbinom{a}{b}=0$ if $b<0$ or $b>a$.

A family $\calF\subseteq\vbinom{V}{k}$ is \emph{intersecting} if
$\dim(F\cap F')\ge1$ for all $F,F'\in\calF$.  It is \emph{trivial} if all
its members contain a common $1$-subspace and \emph{nontrivial} otherwise.

The elements of $\vbinom{V}{1}$ are the projective points.  For
$x\in\vbinom{V}{1}$ define
\[
 d_{\calF}(x)=|\{F\in\calF:x\le F\}|.
\]
For $S\in\vbinom{V}{k}$ define
\[
 d_{\calF}(S)=\sum_{x\in\vbinom{S}{1}}d_{\calF}(x).
\]
If $\calF\ne\vbinom{V}{k}$, its \emph{projective Ore-degree} is
\begin{equation}\label{eq:ore-definition}
 \sigma_{k,q}(\calF)=
 \min\left\{d_{\calF}(S):S\in\vbinom{V}{k}\setminus\calF\right\}.
\end{equation}
When no confusion can arise, we write simply $\sigma(\calF)$.

The use of projective points rather than nonzero vectors is the natural
normalization: every $k$-subspace contains exactly $\qbinom{k}{1}$ projective
points.  Unlike a minimum point-degree, the parameter in
\eqref{eq:ore-definition} permits low-degree points, but controls how such
points can be distributed together inside every nonmember.  For a point $X\in\vbinom{V}{1}$, the \emph{point-star centred at $x$} is
\[
 \calS(X)=\left\{F\in\vbinom{V}{k}:X\le F\right\}.
\]
We also call $\calS(X)$ the \emph{full point-star} to emphasize that it
contains every $k$-subspace of $V$ containing $X$.
Every point $y\ne X$ has degree $\qbinom{n-2}{k-2}$ in $\calS(X)$, and hence
\[
 \sigma_{k,q}(\calS(X))=
 \qbinom{k}{1}\qbinom{n-2}{k-2}.
\]
Our first theorem is the projective $q$-analogue of the Ore-degree
Erd\H{o}s--Ko--Rado theorem.

\begin{thm}[Projective Ore--EKR theorem]\label{thm:ore-ekr}
Let $q$ be a prime power, $k\ge2$ and $n\ge2k+1$.  If
$\calF\subseteq\vbinom{V}{k}$ is intersecting, then
\begin{equation}\label{eq:ore-ekr-bound}
 \sigma_{k,q}(\calF)\le
 \qbinom{k}{1}\qbinom{n-2}{k-2}.
\end{equation}
Equality holds if and only if $\calF$ is a full point-star.
\end{thm}
\begin{rema}[Optimality of the EKR dimension threshold]
	The condition $n\ge2k+1$ cannot be lowered while retaining the equality
	characterization.  If $n=2k$ and $H$ is a $(2k-1)$-subspace, then
	$\vbinom{H}{k}$ is intersecting and is not a full point-star.  Every nonmember
	meets $H$ in dimension $k-1$, whence
	\[
	\sigma_{k,q}\!\left(\vbinom{H}{k}\right)
	=\qbinom{k-1}{1}\qbinom{2k-2}{k-1}
	=\qbinom{k}{1}\qbinom{2k-2}{k-2}.
	\]
	Thus it has exactly the point-star Ore-degree.
\end{rema}

We next state the nontrivial version.  Let $X<M\le V$ with
$\dim X=1$ and $\dim M=k+1$. Define
\begin{equation}\label{eq:HM-definition}
 \calH_1(X,M)=
 \left\{F\in\vbinom{V}{k}:X\le F,\ \dim(F\cap M)\ge2\right\}
 \cup\vbinom{M}{k}.
\end{equation}
Its \emph{outer part} is
\[
 \calH_1^\circ(X,M)=
 \left\{F\in\vbinom{V}{k}:X\le F,\ \dim(F\cap M)\ge2,\
 F\nleq M\right\}.
\]
We will prove that
\[
 \sigma_{k,q}(\calH_1(X,M))=
 \qbinom{k}{1}
 \left(
 \qbinom{n-2}{k-2}
 -q^{k(k-2)}\qbinom{n-k-2}{k-2}
 \right).
\]
When $k=3$, a second construction has the same Ore-degree: for a
$3$-subspace $Z$, let
\[
 \calH_3(Z)=\left\{F\in\vbinom{V}{3}:\dim(F\cap Z)\ge2\right\}.
\]

\begin{thm}[Sharp projective Ore--Hilton--Milner theorem]
\label{thm:ore-hm-sharp}
Let $q$ be a prime power and $k\ge3$.  Assume either
$
 q\ge3\ \text{ and }\ n\ge2k+1,\
\text{or}\ q\ge2\ \text{ and }\ n\ge2k+2.
$
If
$\calF\subseteq\vbinom{V}{k}$ is a nontrivial intersecting family, then
\[
 \sigma_{k,q}(\calF)\le
 \qbinom{k}{1}
 \left(
 \qbinom{n-2}{k-2}
 -q^{k(k-2)}\qbinom{n-k-2}{k-2}
 \right).
\]
Moreover, equality holds if and only if one of the following occurs:
\begin{enumerate}[\rm(i)]
\item for some $X<M$ with $\dim X=1$ and $\dim M=k+1$,
\[
 \calF=\calH_1^\circ(X,M)\cup\calG,
\]
where $\calG\subseteq\vbinom{M}{k}$ and at least one member of $\calG$
does not contain $X$;
\item $k=3$ and, for some $3$-subspace $Z$,
$
 \calF=\calH_3(Z)\ \hbox{or}\
 \calF=\calH_3(Z)\setminus\{Z\}.
$
\end{enumerate}
\end{thm}

\begin{rema}[Binary critical case and optimality of the dimension range]
\label{rema:hm-binary-boundary}
The only unresolved case at the natural boundary $n=2k+1$ is $q=2$.  We
conjecture that, for $q=2$, $k\ge3$, and $n=2k+1$, the inequality and the
complete equality classification in Theorem~\ref{thm:ore-hm-sharp} continue
to hold.  The condition $n\ge2k+1$ in the $q\ge3$ branch of that theorem is
best possible, and the same would then be true for the conjectural binary
extension.  Indeed, if $n=2k$ and $H$ is a $(2k-1)$-subspace, then
$\vbinom{H}{k}$ is nontrivial and intersecting, and
\[
 \sigma_{k,q}\!\left(\vbinom{H}{k}\right)
 =\qbinom{k}{1}\qbinom{2k-2}{k-2}.
\]
This is strictly larger than
\[
 \qbinom{k}{1}
 \left(\qbinom{2k-2}{k-2}-q^{k(k-2)}\right),
\]
which is the proposed right-hand side at $n=2k$.
\end{rema}

\subsection{Vector-space matching}

For two subspaces, trivial intersection is the exact analogue of disjointness.
For three or more subspaces, however, pairwise trivial intersection does not
force their sum to be direct.  For example, three distinct $1$-subspaces of a
$2$-subspace are pairwise trivially intersecting, although their sum has
dimension $2$.  Thus two inequivalent notions occur in finite geometry:
partial-spread matchings, whose members are pairwise trivially intersecting,
and direct-sum matchings.

For a faithful analogue of a set matching in all cardinalities, we use the
rank-additive notion introduced in the vector-space Erd\H{o}s matching problem
\cite{FengShangguanYangZhang}.

\begin{definition}\label{def:direct-matching}
A \emph{direct-sum matching} of size $s$ in
$\calF\subseteq\vbinom{V}{k}$ is a collection
$F_1,\ldots,F_s\in\calF$ satisfying
\[
 \dim(F_1+\cdots+F_s)=sk.
\]
The largest possible $s$ is the \emph{direct-sum matching number} of $\calF$, denoted by
$\nu_{\oplus}(\calF)$.
\end{definition}

For $s=2$, the direct-sum condition is simply $F_1\cap F_2=\{0\}$, so the
intersection theorems above are unchanged.  The distinction is indispensable
for the higher matching result.  Let $W\le V$ have dimension $m$ and define
\begin{equation}\label{eq:B-family}
 \calB_q(V,k;W)=
 \left\{F\in\vbinom{V}{k}:F\cap W\ne\{0\}\right\}.
\end{equation}
No $m+1$ members of this family can have direct sum, because a direct sum would
provide $m+1$ linearly independent nonzero vectors in $W$.  Moreover,
\[
 |\calB_q(V,k;W)|=
 \qbinom{n}{k}-q^{mk}\qbinom{n-m}{k}.
\]
For $m\ge1$, put
\[
 D_q(n,k,m)=
 \qbinom{n-1}{k-1}
 -q^{m(k-1)}\qbinom{n-m-1}{k-1}.
\]
Every point outside $W$ has degree $D_q(n,k,m)$ in
$\calB_q(V,k;W)$, and every nonmember is disjoint from $W$.  Consequently,
\[
 \sigma_{k,q}(\calB_q(V,k;W))=\qbinom{k}{1}D_q(n,k,m).
\]

\begin{thm}[Projective Ore condition for direct-sum matchings]
\label{thm:ore-matching}
Fix a prime power $q$ and integers $k\ge2$ and $s\ge2$.  Suppose that
$n\ge2k+1$ when $s=2$, and that
$n\ge(2s-1)k-s+4$ when $s\ge3$.  If
$\calF\subsetneq\vbinom{V}{k}$ and
\begin{equation}\label{eq:ore-matching-bound}
 \sigma_{k,q}(\calF)>
 \qbinom{k}{1}
 \left(
 \qbinom{n-1}{k-1}
 -q^{(s-1)(k-1)}\qbinom{n-s}{k-1}
 \right),
\end{equation}
then $\nu_{\oplus}(\calF)\ge s$.\end{thm}
\begin{rema}[Projective Ore condition for direct-sum matchings]
The strict inequality is best possible: if $W$ is an $(s-1)$-subspace, then
$\nu_{\oplus}(\calB_q(V,k;W))\le s-1$, while its Ore-degree equals the
right-hand side of \eqref{eq:ore-matching-bound}.
In fact, this construction has direct-sum matching number exactly $s-1$
whenever $n\ge (s-1)k$.
\end{rema}

Theorem~\ref{thm:ore-matching} is the direct vector-space counterpart of the
Ore matching theorem in \cite{BaloghPalmerRaeisi}.  Its proof combines a
projective incidence inequality with the Hilton--Milner-type stability theorem
for the vector-space Erd\H{o}s matching problem
\cite{FengShangguanYangZhang}.  For $s\ge3$, its dimension threshold is the
one in that stability theorem; the case $s=2$ is covered by
Theorem~\ref{thm:ore-ekr}.

The matching theorem immediately yields the same multicolour consequence as
in the set case.

\begin{definition}\label{def:colouring}
Let $c$ be a positive integer and write $[c]=\{1,\ldots,c\}$.  A
\emph{$c$-colouring} of a family
$\calF\subseteq\vbinom{V}{k}$ is a map
\[
 \chi:\calF\longrightarrow[c].
\]
The members of $\chi^{-1}(i)$ are said to have colour $i$.  A direct-sum
matching $F_1,\ldots,F_s$ in $\calF$ is \emph{monochromatic} if
\[
 \chi(F_1)=\cdots=\chi(F_s).
\]
If their common colour is $i$, we call them a monochromatic direct-sum
matching of colour $i$ and size $s$.
\end{definition}

\begin{cor}[A direct-sum Ramsey consequence]\label{cor:ramsey}
Fix $q$, $k\ge2$, and positive integers $a_1,\ldots,a_c$.  Set
\[
 M=\sum_{i=1}^{c}(a_i-1).
\]
If $M=0$, assume $n\ge k$; if $M=1$, assume $n\ge2k+1$; and if $M\ge2$,
assume $n\ge(2M+1)k-M+3$.  If
$\calF\subsetneq\vbinom{V}{k}$ satisfies
\begin{equation}\label{eq:ramsey-ore}
 \sigma_{k,q}(\calF)>
 \qbinom{k}{1}\left(
 \qbinom{n-1}{k-1}
 -q^{M(k-1)}\qbinom{n-M-1}{k-1}
 \right),
\end{equation}
then every $c$-colouring of $\calF$ contains, for some $i\in[c]$, a
monochromatic direct-sum matching of size $a_i$ in colour $i$.
\end{cor}

Theorems \ref{thm:ore-ekr} and \ref{thm:ore-hm-sharp} use the incidence inequality
together with sharp cardinality and structural theorems; at the boundary
$n=2k+1$, the latter also uses covering-number estimates, a
disjoint-subspace averaging lemma, and the classification of large
intersecting families of planes.
Theorem~\ref{thm:ore-matching} instead depends on the recent stability theorem
for the vector-space Erd\H{o}s matching problem.

The rest of the paper is organized as follows.  Section~\ref{sec:prelim}
contains the Gaussian coefficient identities, the incidence inequality, and
the Ore-degree calculations for the extremal constructions.
Section~\ref{sec:inputs} records the external structural results used in
the proofs.  Sections~\ref{sec:ekr-hm} and \ref{sec:matching} prove the main
theorems.

\section{Counting and the projective Ore inequality}\label{sec:prelim}

We first record standard Gaussian coefficient facts.

\begin{lemma}\label{lem:gaussian}
Let $N\ge r\ge1$ and let $a\ge0$ be fixed.  Empty products are understood
to be $1$.
\begin{enumerate}[\rm(i)]
\item
\[
 \qbinom{N}{r}=\qbinom{N-1}{r-1}+q^r\qbinom{N-1}{r}.
\]
\item For every positive integer $h$ with $N-h\ge r$,
\begin{equation}\label{eq:telescoping-gaussian}
 \qbinom{N}{r}-q^{hr}\qbinom{N-h}{r}
 =\sum_{j=0}^{h-1}q^{jr}\qbinom{N-j-1}{r-1}.
\end{equation}
\item As $n\to\infty$,
\[
 \qbinom{n-a}{r}
 =\frac{q^{r(n-a-r)}}{\prod_{i=1}^{r}(1-q^{-i})}
 \bigl(1+O(q^{-n})\bigr).
\]
\item If $L$ is an $\ell$-subspace of an $N$-dimensional vector space, then
the number of $r$-subspaces $U$ satisfying $U\cap L=\{0\}$ is
\[
 q^{\ell r}\qbinom{N-\ell}{r}.
\]
\item
\[
 q^{r(N-r)}\le\qbinom{N}{r}<4q^{r(N-r)}.
\]
\end{enumerate}
Statements (iii)--(v) remain valid for $r=0$.
\end{lemma}

\begin{proof}
Part (i) is the usual Gaussian recurrence, and iterating it gives (ii).  Part
(iii) follows from
\[
 \qbinom{n-a}{r}
 =q^{r(n-a-r)}
 \prod_{j=0}^{r-1}
 \frac{1-q^{-(n-a-j)}}{1-q^{-(r-j)}}.
\]
For (iv), choose a complement $W$ of $L$.  Every $r$-subspace disjoint from
$L$ is the graph of a linear map from an $r$-subspace of $W$ to $L$.
There are $\qbinom{N-\ell}{r}$ choices for the domain and $q^{\ell r}$
choices for the map.  Finally, the product formula shows the lower bound in
(v).  For the upper bound, the numerator factors are smaller than $1$, while
\[
 \prod_{i=1}^{r}(1-q^{-i})
 \ge \frac12\frac34
 \left(1-\sum_{i=3}^{\infty}2^{-i}\right)
 =\frac9{32}>\frac14.
\]
\end{proof}

The following finite estimate replaces all asymptotic separation arguments.

\begin{lemma}\label{lem:gaussian-difference}
If $M,t,r$ are positive integers and $M\ge t+r$, then
\begin{align*}
 \left(1-4q^{-(M-t-r+2)}\right)
 \qbinom{t}{1}\qbinom{M-1}{r-1}
 \le \qbinom{M}{r}-q^{tr}\qbinom{M-t}{r}
 \le \qbinom{t}{1}\qbinom{M-1}{r-1}.
\end{align*}
\end{lemma}

\begin{proof}
By \eqref{eq:telescoping-gaussian}, the middle expression is
\[
 \sum_{j=0}^{t-1}q^{jr}\qbinom{M-j-1}{r-1}.
\]
For $0\le j\le t-1$, division by $\qbinom{M-1}{r-1}$ and the product
formula give
\[
 q^j\prod_{u=0}^{r-2}
 \frac{1-q^{-(M-j-1-u)}}{1-q^{-(M-1-u)}}.
\]
Every factor in the product is at most $1$, proving the upper bound after
summing $q^j$.  For the lower bound, each factor is at least
$1-2q^{-(M-j-1-u)}$.  The product is therefore at least
\[
 1-2\sum_{u=0}^{r-2}q^{-(M-j-1-u)}
 \ge1-4q^{-(M-t-r+2)}.
\]
Summing again proves the result.
\end{proof}

Put
\[
 G=\qbinom{n}{k},\qquad R=\qbinom{n-1}{k-1}.
\]
Counting incident pairs $(x,F)$ with $x\le F$ gives
\begin{equation}\label{eq:incidence-identity}
 \qbinom{n}{1}R=\qbinom{k}{1}G.
\end{equation}

\begin{obs}\label{obs:monotone}
If $\calA\subseteq\calB\subsetneq\vbinom{V}{k}$, then
\[
 \sigma_{k,q}(\calA)\le\sigma_{k,q}(\calB).
\]
\end{obs}

\begin{proof}
Every nonmember of $\calB$ is a nonmember of $\calA$, and
$d_{\calA}(S)\le d_{\calB}(S)$ for every such $S$.  Therefore
\[
 \sigma(\calA)
 \le\min_{S\notin\calB}d_{\calA}(S)
 \le\min_{S\notin\calB}d_{\calB}(S)=\sigma(\calB).
\]
\end{proof}

The next lemma is the projective counterpart of the basic quadratic
Ore-degree inequality for uniform set systems.

\begin{lemma}\label{lem:ore-size}
If $\calF\subsetneq\vbinom{V}{k}$, then
\begin{equation}\label{eq:ore-size}
 |\calF|\ge
 \frac{\qbinom{n}{1}}{\qbinom{k}{1}^{2}}\,\sigma_{k,q}(\calF).
\end{equation}
Equality can hold only if all projective points have the same degree in
$\calF$.
\end{lemma}

\begin{proof}
Write $f=|\calF|$ and $d_x=d_{\calF}(x)$.  Summing the defining Ore
inequality over all nonmembers of $\calF$ gives
\begin{align*}
 \sigma(\calF)(G-f)
 \le \sum_{S\notin\calF}\sum_{x\le S}d_x=\sum_{x\in\vbinom{V}{1}}d_x(R-d_x).
\end{align*}
Since each member of $\calF$ contains $\qbinom{k}{1}$ points,
\[
 \sum_xd_x=\qbinom{k}{1}f.
\]
‌By the Cauchy–Schwarz inequality,
\[
 \sum_xd_x^2\ge
 \frac{\qbinom{k}{1}^{2}f^2}{\qbinom{n}{1}},
\]
with equality only when all $d_x$ are equal.  Hence, using
\eqref{eq:incidence-identity},
\begin{align*}
 \sigma(\calF)(G-f)
 \le \qbinom{k}{1}Rf
 -\frac{\qbinom{k}{1}^{2}f^2}{\qbinom{n}{1}}=\frac{\qbinom{k}{1}^{2}}{\qbinom{n}{1}}f(G-f).
\end{align*}
Since $G-f>0$, division proves \eqref{eq:ore-size}.
\end{proof}

\begin{lemma}\label{lem:ore-disjoint-average}
Let $n\ge2k$, 
$\emptyset\ne\calF\subseteq\vbinom{V}{k}$ be intersecting and put
\[
 p=\qbinom{k}{1},\qquad N=\qbinom{n}{1}.
\]
Then
\begin{equation}\label{eq:ore-disjoint-average}
 \sigma_{k,q}(\calF)\le
 \frac{p(p-1)}{N-p}\bigl(|\calF|-1\bigr).
\end{equation}
\end{lemma}

\begin{proof}
Fix $F_0\in\calF$ and average $d_{\calF}(S)$ over all $k$-subspaces
$S$ disjoint from $F_0$.  Every such $S$ is a nonmember of $\calF$.
Their number is
\[
 A=q^{k^2}\qbinom{n-k}{k}.
\]
For each point $x\nleq F_0$, the number of these subspaces containing $x$
is
\[
 B=q^{k(k-1)}\qbinom{n-k-1}{k-1}.
\]
Consequently,
\[
 A\,\sigma_{k,q}(\calF)
 \le B\sum_{x\nleq F_0}d_{\calF}(x).
\]
Now
\[
 \sum_{x\le F_0}d_{\calF}(x)
 =\sum_{F\in\calF}\qbinom{\dim(F\cap F_0)}{1}
 \ge |\calF|+p-1,
\]
because $F_0$ contributes $p$ and every other member contributes at
least one.  Since $\sum_xd_{\calF}(x)=p|\calF|$, it follows that
\[
 \sum_{x\nleq F_0}d_{\calF}(x)
 \le(p-1)(|\calF|-1).
\]
Finally, the adjacent-ratio identity and
$N-p=q^k\qbinom{n-k}{1}$ give
\[
 \frac BA=
 q^{-k}\frac{\qbinom{n-k-1}{k-1}}{\qbinom{n-k}{k}}
 =\frac{p}{N-p}.
\]
Substitution proves \eqref{eq:ore-disjoint-average}.
\end{proof}

\subsection{The full point-star and Hilton--Milner constructions}

\begin{prop}\label{prop:star-degree}
Let $X\in\vbinom{V}{1}$ and $n\ge k+1$.  Then
\begin{equation}\label{eq:star-ore}
 \sigma_{k,q}(\calS(X))=
 \qbinom{k}{1}\qbinom{n-2}{k-2}.
\end{equation}
If $\calF$ is a proper subfamily of $\calS(X)$, then
\[
 \sigma_{k,q}(\calF)<
 \qbinom{k}{1}\qbinom{n-2}{k-2}.
\]
\end{prop}

\begin{proof}
If $y\ne X$ is a point, then the members of $\calS(X)$ containing $y$ are
exactly the $k$-subspaces containing the $2$-subspace $X+y$.  Hence
$d_{\calS(X)}(y)=\qbinom{n-2}{k-2}$.  Every nonmember of the star avoids
$X$, proving \eqref{eq:star-ore}.

Now let $E\in\calS(X)\setminus\calF$.  Choose a point $y\le E$ with
$y\ne X$, and choose a $k$-subspace $S$ containing $y$ but not $X$.  Then
$S\notin\calF$ and
\[
 d_{\calF}(y)<\qbinom{n-2}{k-2},
\]
while every other point of $S$ has degree at most
$\qbinom{n-2}{k-2}$.  Thus $d_{\calF}(S)$ is strictly smaller than the
right-hand side of \eqref{eq:star-ore}.
\end{proof}

\begin{prop}\label{prop:hm-degree}
Let $X<M\le V$ with $\dim X=1$ and $\dim M=k+1$.  If $n\ge2k+1$, then
\begin{equation}\label{eq:HM-ore-exact}
 \sigma_{k,q}(\calH_1(X,M))=
 \qbinom{k}{1}
 \left(
 \qbinom{n-2}{k-2}
 -q^{k(k-2)}\qbinom{n-k-2}{k-2}
 \right).
\end{equation}
\end{prop}

\begin{proof}
Fix a point $y\not\le M$.  There are $\qbinom{n-2}{k-2}$ $k$-subspaces
containing $X+y$.  Such a subspace fails the first condition in
\eqref{eq:HM-definition} precisely when its intersection with $M$ is $X$.
In the quotient $V/(X+y)$ this is equivalent to choosing a $(k-2)$-subspace
disjoint from the $k$-subspace $(M+X+y)/(X+y)$.  By
Lemma~\ref{lem:gaussian}(iv), the number excluded is
\[
 q^{k(k-2)}\qbinom{n-k-2}{k-2}.
\]
No member of $\vbinom{M}{k}$ contains $y$, and hence
\begin{equation}\label{eq:HM-outside-point-degree}
 d_{\calH_1(X,M)}(y)=
 \qbinom{n-2}{k-2}
 -q^{k(k-2)}\qbinom{n-k-2}{k-2}.
\end{equation}
If $y\le M$ and $y\ne X$, then every $k$-subspace containing $X+y$ belongs
to $\calH_1(X,M)$, so its degree is at least
$\qbinom{n-2}{k-2}$, which is at least the right-hand side of
\eqref{eq:HM-outside-point-degree}.  The degree of $X$ is still
larger.  Finally, $n\ge2k+1$ permits a $k$-subspace $S$ disjoint from $M$.
It is a nonmember, all its points have the degree in
\eqref{eq:HM-outside-point-degree}, and summing over its
$\qbinom{k}{1}$ points proves \eqref{eq:HM-ore-exact}.
\end{proof}

\begin{prop}
\label{prop:h3-degree}
Let $k=3$, let $Z$ be a $3$-subspace, and assume $n\ge6$.  Then
\begin{equation}\label{eq:H3-ore}
 \sigma_{3,q}(\calH_3(Z))=\qbinom{3}{1}^{2}.
\end{equation}
Moreover,
\[
 \qbinom{n-2}{1}-q^3\qbinom{n-5}{1}=\qbinom{3}{1}.
\]
\end{prop}

\begin{proof}
For a point $y\not\le Z$, a member of $\calH_3(Z)$ containing $y$ is
uniquely of the form $L+y$, where $L$ is a $2$-subspace of $Z$.  Thus
$d_{\calH_3(Z)}(y)=\qbinom{3}{2}=\qbinom{3}{1}$.  Points in $Z$ have larger
degree, and a $3$-subspace disjoint from $Z$ witnesses
\eqref{eq:H3-ore}.  Finally,
\[
 \qbinom{n-2}{1}-q^3\qbinom{n-5}{1}
 =1+q+q^2=\qbinom{3}{1}.
\]
\end{proof}

\subsection{The matching construction}

\begin{prop}\label{prop:B-degree}
Let $W\le V$ with  $\dim W=m\ge1$, and assume $n\ge k+m$.  Then
\begin{enumerate}[\rm(i)]
\item
$|\calB_q(V,k;W)|=\qbinom{n}{k}-q^{mk}\qbinom{n-m}{k}$;
\item $\nu_{\oplus}(\calB_q(V,k;W))\le m$;
\item
$\sigma_{k,q}(\calB_q(V,k;W))=\qbinom{k}{1}D_q(n,k,m)$.
\end{enumerate}
If $n\ge mk$, then the matching number in (ii) is exactly $m$.
\end{prop}

\begin{proof}
Part (i) follows from Lemma~\ref{lem:gaussian}(iv), since the complement of
$\calB_q(V,k;W)$ is the family of $k$-subspaces disjoint from $W$.

Suppose, for a contradiction, that
$F_1,\ldots,F_{m+1}\in\calB_q(V,k;W)$ form a direct-sum matching.  For each
$i\in\{1,\ldots,m+1\}$, choose $0\ne w_i\in F_i\cap W$.  The direct-sum
condition would force $w_1,\ldots,w_{m+1}$ to be linearly independent in
the $m$-dimensional space $W$, a contradiction.  This proves (ii).

Fix a point $x\not\le W$.  Among all $k$-subspaces containing $x$, those
which avoid $W$ correspond in $V/x$ to $(k-1)$-subspaces disjoint from the
$m$-subspace $(W+x)/x$.  Therefore
\[
 d_{\calB_q(V,k;W)}(x)
 =\qbinom{n-1}{k-1}
 -q^{m(k-1)}\qbinom{n-m-1}{k-1}
 =D_q(n,k,m).
\]
Every point of $W$ has degree $\qbinom{n-1}{k-1}$.  Every nonmember of
$\calB_q(V,k;W)$ is disjoint from $W$, so all its $\qbinom{k}{1}$ points have degree
$D_q(n,k,m)$.  This proves (iii).

Finally, suppose $n\ge mk$.  Choose independent points
$x_1,\ldots,x_m$ spanning $W$ and a complement $U$ of $W$.  The space $U$
contains subspaces $A_1,\ldots,A_m$ of dimension $k-1$ whose sum is direct.
Then $F_i=x_i+A_i$ belong to $\calB_q(V,k;W)$ and their sum is direct.
\end{proof}

\section{Structural inputs}\label{sec:inputs}

This section records the external extremal results used below and briefly
mentions the recent low-dimensional structure theorem relevant to the
critical range.  We first need the complete vector-space Hilton--Milner
theorem.  The remaining binary boundary case left open in
\cite{BlokhuisEtAl} was settled in \cite{WangXuZhang}.

\begin{thm}[Blokhuis et al.; Wang--Xu--Zhang]
\label{thm:vector-HM}
Let $k\ge3$ and $n\ge2k+1$.  If
$\calF\subseteq\vbinom{V}{k}$ is nontrivial and intersecting, then
\[
 |\calF|\le
 \qbinom{n-1}{k-1}
 -q^{k(k-1)}\qbinom{n-k-1}{k-1}+q^k.
\]
\end{thm}

Ihringer and Kupavskii recently proved that large $t$-intersecting families
in the natural range $n\ge2k+1$ admit a governing low-dimensional structure
\cite{IhringerKupavskii}.  Their simplification framework explains why the
natural boundary is $n=2k+1$.  For the exact numerical threshold needed
below, however, we use the more specialized covering-number estimates of
Cao, Lv, Wang and Zhou.  At $q=2$ and $n=2k+1$, the presently available
remainder estimate in \cite{IhringerKupavskii} is not small enough by itself
to imply the sharp Ore threshold; this is the binary critical case recorded
in Remark~\ref{rema:hm-binary-boundary}.

For the sharp Ore--Hilton--Milner theorem we use the $t=1$ case of the large
nontrivial intersection theorem of Cao, Lv, Wang and Zhou
\cite{CaoLvWangZhou}.

For an intersecting family $\calF$, its \emph{covering number} is
\[
 \tau(\calF)=\min\{\dim T:T\le V,\ T\cap F\ne\{0\}
                    \text{ for every }F\in\calF\}.
\]
Thus $\calF$ is nontrivial precisely when $\tau(\calF)\ge2$.

Let $\calH_2(X,M,C)$ denote Family II of \cite{CaoLvWangZhou}, where
$X\le M\le C\le V$ and
\[
 \dim X=1,\qquad \dim M=k,\qquad
 \dim C=c\in\{k+1,\ldots,2k-1,n\}.
\]
We only use that $\calH_2(X,M,C)=\calH_1(X,C)$ when $c=k+1$, and the
size formula below.
For a $3$-subspace $Z$, define for arbitrary $k\ge3$
\[
 \calH_3^{(k)}(Z)=
 \left\{F\in\vbinom{V}{k}:\dim(F\cap Z)\ge2\right\}.
\]
The relevant sizes are
\begin{align*}
 h_2(n,k,c)=&\ \qbinom{n-1}{k-1}
 -q^{(k-1)^2}\qbinom{n-k}{k-1}
 +q^{(k-1)^2}\qbinom{n-c}{2k-c-1}
 +q^k\qbinom{c-k}{1},\\
 h_3(n,k)=&\ \qbinom{3}{1}\qbinom{n-2}{k-2}
 -q\qbinom{2}{1}\qbinom{n-3}{k-3}.
\end{align*}

\begin{thm}[Cao--Lv--Wang--Zhou]\label{thm:CLWZ-structure}
Let $k\ge4$ and $n\ge2k+3$.  Suppose that
$\calF\subseteq\vbinom{V}{k}$ is a maximal nontrivial intersecting family and
\[
 |\calF|\ge f_q(n,k):=
 \qbinom{k-1}{1}\qbinom{n-2}{k-2}
 -q\qbinom{k-1}{2}\qbinom{n-3}{k-3}.
\]
Then either
\begin{enumerate}[\rm(i)]
\item $\calF=\calH_2(X,M,C)$ for some
$c\in\{k+1,\ldots,2k-1,n\}$; or
\item $k=4$ and $\calF=\calH_3^{(4)}(Z)$.
\end{enumerate}
\end{thm}

For the matching theorem, we use the direct-sum matching stability theorem of
Feng, Shangguan, Yang and Zhang \cite{FengShangguanYangZhang}.  We reproduce
the recursive extremal family needed in the proof.  Fix a point $I$ and a
$k$-subspace $E$ with $I\cap E=\{0\}$, and put
\begin{equation}\label{eq:linear-HM-base}
 \calL_q(n,k,1)=
 \left\{F\in\vbinom{V}{k}:I\le F,\ F\cap E\ne\{0\}\right\}
 \cup\vbinom{I+E}{k}.
\end{equation}
For $m\ge2$, write $V_n=V_{n-1}\oplus I_n$, where $I_n$ is a point, and
let $\pi:V_n\to V_{n-1}$ be the canonical projection.  Define recursively
\begin{equation}\label{eq:linear-HM-recursive}
 \calL_q(n,k,m)=
 \{F\in\vbinom{V_n}{k}:I_n\le F\}
 \cup
 \{F\in\vbinom{V_n}{k}:F\cap I_n=\{0\},\
      \ \pi(F)\in\calL_q(n-1,k,m-1)\}.
\end{equation}
Let $\ell_q(n,k,m)=|\calL_q(n,k,m)|$.  Since each lift of a
$k$-subspace of $V_{n-1}$ is the graph of a linear map into $I_n$,
\begin{equation}\label{eq:ell-recurrence}
 \ell_q(n,k,m)=\qbinom{n-1}{k-1}+q^k\ell_q(n-1,k,m-1)
 \qquad(m\ge2).
\end{equation}

\begin{thm}[Feng--Shangguan--Yang--Zhang]\label{thm:linear-stability}
Let $q$ be a prime power and let $k\ge2$, $m\ge1$.  Suppose that
\[
 n\ge(2m+1)k-m+3.
\]
If $\calF\subseteq\vbinom{V}{k}$ satisfies $\nu_{\oplus}(\calF)\le m$ and
is not contained in any family $\calB_q(V,k;W)$ with $\dim W=m$, then
\[
 |\calF|\le\ell_q(n,k,m).
\]
\end{thm}

\section{Ore--EKR and Ore--Hilton--Milner}\label{sec:ekr-hm}

We first isolate an exact comparison for the EKR argument.

\begin{lemma}\label{lem:ekr-separation}
Let $q$ be a prime power, $k\ge3$ and $n\ge2k+1$.  Then
\[
 \frac{\qbinom{n}{1}}{\qbinom{k}{1}}\qbinom{n-2}{k-2}>
 \qbinom{n-1}{k-1}
 -q^{k(k-1)}\qbinom{n-k-1}{k-1}+q^k.
\]
\end{lemma}

\begin{proof}
Put $Q=\qbinom{n-2}{k-2}$.  The upper bound in
Lemma~\ref{lem:gaussian-difference}, with
$(M,t,r)=(n-1,k,k-1)$, gives
\[
 \qbinom{n-1}{k-1}
 -q^{k(k-1)}\qbinom{n-k-1}{k-1}
 \le\qbinom{k}{1}Q.
\]
On the other hand,
\[
 \frac{\qbinom{n}{1}}{\qbinom{k}{1}}Q
 =\frac{q^n-1}{q^k-1}Q
 >q^{n-k}Q\ge q^{k+1}Q.
\]
Since $\qbinom{k}{1}<q^k$ and $Q\ge1$,
\[
 \qbinom{k}{1}Q+q^k<q^k(Q+1)\le q^{k+1}Q.
\]
Combining the three inequalities proves the claim.
\end{proof}

\begin{proof}[Proof of Theorem~\ref{thm:ore-ekr}]
Suppose first that $\calF$ is contained in a full point-star.
Proposition~\ref{prop:star-degree} gives \eqref{eq:ore-ekr-bound}, with
equality only for the full  point-star.

Suppose next that $\calF$ is not contained in a full point-star. Then it is
nontrivial.  If $k\ge3$ and
$\sigma(\calF)\ge\qbinom{k}{1}\qbinom{n-2}{k-2}$, then
Lemma~\ref{lem:ore-size} gives
\[
 |\calF|\ge
 \frac{\qbinom{n}{1}}{\qbinom{k}{1}}\qbinom{n-2}{k-2}.
\]
This contradicts Theorem~\ref{thm:vector-HM} and
Lemma~\ref{lem:ekr-separation}.

It remains to treat $k=2$.  Choose distinct $A,B\in\calF$ and put
$X=A\cap B$.  Nontriviality provides $C\in\calF$ not containing $X$.
The two points $A\cap C$ and $B\cap C$ are distinct, so
$C\le A+B$.  If $D\in\calF$ does not contain $X$, the same argument gives
$D\le A+B$; if $X\le D$, then $D$ meets $C$ in a point different from
$X$, and again $D\le A+B$.  Thus every member of $\calF$ lies in the fixed
$3$-space $A+B$, and hence $|\calF|\le\qbinom{3}{1}$.  But
$\sigma(\calF)\ge\qbinom{2}{1}$ would imply
\[
 |\calF|\ge\frac{\qbinom{n}{1}}{\qbinom{2}{1}}
 >q^{n-2}\ge q^3>\qbinom{3}{1},
\]
a contradiction.  Therefore every nontrivial family has strictly smaller
Ore-degree than a full point-star, and the theorem follows.
\end{proof}

\begin{lemma}\label{lem:HM-separation}
Let $q$ be a prime power, $k\ge4$ and $n\ge2k+3$.  For every
$c\in\{k+2,\ldots,2k-1,n\}$,
\begin{align*}
 \frac{\qbinom{n}{1}}{\qbinom{k}{1}}
 \left(
 \qbinom{n-2}{k-2}
 -q^{k(k-2)}\qbinom{n-k-2}{k-2}
 \right)
 >\max\{f_q(n,k),h_3(n,k),h_2(n,k,c)\}.
\end{align*}
\end{lemma}

\begin{proof}
Put $Q=\qbinom{n-2}{k-2}$.  Lemma~\ref{lem:gaussian-difference}, with
$(M,t,r)=(n-2,k,k-2)$, and the identity
$\qbinom{n-3}{k-3}=\frac{\qbinom{k-2}{1}}{\qbinom{n-2}{1}}Q$ give
\begin{align*}
 \frac{\qbinom{n}{1}}{\qbinom{k}{1}}
 \left(Q-q^{k(k-2)}\qbinom{n-k-2}{k-2}\right)>(1-4q^{-5})q^2\qbinom{k-2}{1}Q,
\end{align*}
because $n-2k+2\ge5$ and
$\qbinom{n}{1}/\qbinom{n-2}{1}>q^2$.
Moreover,
\[
 f_q(n,k)\le\qbinom{k-1}{1}Q,
 \qquad h_3(n,k)\le\qbinom{3}{1}Q
 \le\qbinom{k-1}{1}Q.
\]
For $h_2(n,k,c)$, another application of Lemma~\ref{lem:gaussian-difference}, now
with $(M,t,r)=(n-1,k-1,k-1)$, yields
\[
 \qbinom{n-1}{k-1}-q^{(k-1)^2}\qbinom{n-k}{k-1}
 \le\qbinom{k-1}{1}Q.
\]
If $k+2\le c\le2k-1$ and $r=2k-c-1$, then $0\le r\le k-3$ and
\[
 q^{(k-1)^2}\qbinom{n-c}{r}
 <4q^{(k-1)^2+(k-3)(n-2k+1)}
 \le4q^{k-5}Q.
\]
For $c=n$ this term is zero.  Also
\[
 q^k\qbinom{c-k}{1}<Q.
\]
Indeed, for $c\le2k-1$ its left-hand side is smaller than $q^{2k-1}$,
whereas $Q\ge q^{(k-2)(k+3)}$; for $c=n$ it is smaller than $q^n$,
whereas $Q\ge q^{(k-2)(n-k)}>q^n$.  Consequently,
\[
 h_2(n,k,c)<
 \left(\qbinom{k-1}{1}+4q^{k-5}+1\right)Q.
\]
Finally, if
\[
 \Delta=(1-4q^{-5})q^2\qbinom{k-2}{1}
 -\qbinom{k-1}{1}-4q^{k-5}-1,
\]
then
\begin{align*}
 q^3(q-1)\Delta
 &=q^{k-1}(q^4-q^3-4)-q^5-q^4+2q^3+4\\
 &\ge q^7-q^6-q^5-q^4-2q^3+4\\
 &=q^3(q^4-q^3-q^2-q-2)+4>0.
\end{align*}
The asserted strict inequalities follow.
\end{proof}

\begin{lemma}\label{lem:HM-critical-comparison}
Let $q\ge3$ and $k\ge4$.  Put
\[
 p=\qbinom{k}{1},\quad N=\qbinom{2k+1}{1},\quad
 Q=\qbinom{2k-1}{k-2},\quad R=\qbinom{2k-2}{k-3},
\]
\[
 D=Q-q^{k(k-2)}\qbinom{k-1}{1},
 \qquad L=\frac{N}{p}D.
\]
For \(3\le l\le k-1\), define
\begin{align*}
 U_0&=Q+q(q+1)\qbinom{k-1}{1}pR,\\
 U_l&=\qbinom{l-1}{1}Q
 +q^{l-1}\qbinom{k-l+1}{1}pR
 +q^k\qbinom{2k+1-l}{k-l+1},\\
 U_3^*&=\bigl((q+1)p^2+q^4p\bigr)R.
\end{align*}
Then
\begin{equation}\label{eq:critical-list}
 \max\bigl\{U_0,\ U_l,\ p^k,\ h_3(2k+1,k),\
 h_2(2k+1,k,c)\bigr\}<L
\end{equation}
for every \(3\le l\le k-1\) and
\(c\in\{k+2,\ldots,2k-1,2k+1\}\).  Moreover,
\[
 U_3^*<L
\]
unless \((q,k)=(3,4)\).  In that exceptional case,
\[
 p=40,\quad N=9841,\quad R=364,\quad D=14170,\quad
 L=\frac{13944697}{4},\quad U_3^*=3508960,
\]
and
\begin{equation}\label{eq:critical-exception-average}
 \frac{p(p-1)(U_3^*-1)}{N-p}<pD.
\end{equation}
\end{lemma}

\begin{proof}
The Grassmann intersection-number formula gives
\[
 D=\sum_{j=1}^{k-2}
 q^{(k-j)(k-2-j)}
 \qbinom{k}{j}\qbinom{k-1}{k-2-j}.
\]
Keeping the term \(j=1\), we obtain
\[
 L\ge
 Nq^{(k-1)(k-3)}\qbinom{k-1}{2}.
\]
Furthermore,
\[
 \frac{q^{(k-1)(k-3)}\qbinom{k-1}{2}}{R}
 =\prod_{i=3}^{k-1}
 \frac{1-q^{-i}}{1-q^{-(k-1+i)}}
 >\beta_q:=1-\frac1{q^2(q-1)}.
\]
Thus
\begin{equation}\label{eq:critical-L-lower}
 L>\beta_qNR,\qquad \beta_q\ge\frac{17}{18}.
\end{equation}

We use
\[
 N>q^{2k-1}(q+1),\qquad
 R>q^{(k-3)(k+1)},\qquad
 \frac QR<\frac{q^{k+1}}{1-q^{-2}},
\]
together with
\(\qbinom{a}{b}<2q^{b(a-b)}\) for \(q\ge3\).  The latter follows by
factoring out \(q^{b(a-b)}\): the remaining denominator is bounded below
by
\[
 \prod_{i=1}^{\infty}(1-q^{-i})
 >1-\sum_{i=1}^{\infty}q^{-i}\ge\frac12.
\]
For \(U_0\), the two summands satisfy
\[
 \frac{Q}{NR}<\frac1{32},\qquad
 \frac{q(q+1)\qbinom{k-1}{1}p}{N}<\frac34.
\]
The same substitutions give
\begin{align}
 \frac{U_0}{NR}
 &<\frac34+\frac1{32}=\frac{25}{32},\label{eq:critical-U0}\\
 \max_{3\le l\le k-1}\frac{U_l}{NR}
 &<\frac9{64}+\frac3{16}+\frac12
 =\frac{53}{64},\label{eq:critical-Ul}\\
 \frac{h_3(2k+1,k)}{NR}
 &<\frac{27}{64}.\label{eq:critical-h3}
\end{align}
For example, the three summands in \eqref{eq:critical-Ul} follow,
respectively, from
\[
 \frac{q^2}{(q^2-1)^2},\qquad
 \frac{q}{(q-1)^2(q+1)},\qquad
 \frac{2q^{4-k}}{q+1},
\]
each of which is largest at \(q=3,k=4\).  Likewise,
\eqref{eq:critical-h3} follows from
\[
 \frac{\qbinom{3}{1}Q}{NR}
 <\frac{q^{7-k}}{(q^2-1)^2}\le\frac{27}{64}.
\]

For the Family II term, the Gaussian difference estimate gives
\[
 \qbinom{2k}{k-1}
 -q^{(k-1)^2}\qbinom{k+1}{k-1}
 \le\qbinom{k-1}{1}Q.
\]
If \(c\le2k-1\), put \(s=2k-c-1\), so
\(0\le s\le k-3\) and
\(\qbinom{2k+1-c}{2k-c-1}=\qbinom{s+2}{2}\);
for \(c=2k+1\) this term is zero.  The three parts of \(h_2(2k+1,k,c)/(NR)\)
are therefore smaller than
\[
 \frac{27}{64},\qquad
 \frac{2}{q(q+1)}\le\frac16,\qquad
 \frac{q^{-k^2+2k+5}}{(q-1)(q+1)}\le\frac1{216},
\]
respectively.  Hence
\[
 \frac{h_2(2k+1,k,c)}{NR}
 <\frac{1025}{1728}<\frac{17}{18}.
\]

Next,
\[
 \frac{p^k}{NR}
 <\frac{q^4}{(q+1)(q-1)^k}.
\]
This is smaller than \(17/18\) when \(q=3,k\ge5\) or
\(q\ge4,k\ge4\).  For \((q,k)=(3,4)\), direct substitution gives
\[
 \frac{p^k}{NR}=\frac{40^4}{9841\cdot364}<\frac{17}{18}.
\]
Finally,
\[
 \frac{U_3^*}{NR}
 <\frac{q}{(q-1)^2}+\frac{q^{5-k}}{q^2-1}.
\]
For \(q=3,k\ge5\), the right-hand side is at most \(7/8\);
for \(q\ge4,k\ge4\), it is at most \(32/45\).  This proves all
comparisons with \eqref{eq:critical-L-lower}, except
\((q,k)=(3,4)\).

The displayed exceptional values follow by direct evaluation.  In addition,
\[
 pD(N-p)-p(p-1)(U_3^*-1)=81230760>0,
\]
which proves \eqref{eq:critical-exception-average}.
\end{proof}

\begin{lemma}\label{lem:HM-critical-structure}
Let $q\ge3$, \(k\ge4\), and \(n=2k+1\).  Put
\[
 p=\qbinom{k}{1},\qquad
 D=\qbinom{n-2}{k-2}
 -q^{k(k-2)}\qbinom{n-k-2}{k-2}.
\]
If \(\calF\subseteq\vbinom{V}{k}\) is a maximal nontrivial intersecting
family and
\[
 \sigma_{k,q}(\calF)\ge pD,
\]
then \(\calF=\calH_1(X,M)\) for some \(X<M\le V\) with
\((\dim X,\dim M)=(1,k+1)\).
\end{lemma}

\begin{proof}
Use the notation of Lemma~\ref{lem:HM-critical-comparison}.  By
Lemma~\ref{lem:ore-size},
\[
 |\calF|\ge L.
\]
The covering-number analysis in
\cite[Lemmas~3.3--3.8]{CaoLvWangZhou} is valid under \(2k\le n\);
we specialize it to \(t=1\).

Suppose first that \(\tau(\calF)=2\), and let \(\calT\) be the family of
all \(2\)-dimensional covers.  The alternatives in
\cite[Lemmas~3.3--3.7]{CaoLvWangZhou} are bounded by \(U_0\), by some
\(U_l\) with \(3\le l\le k-1\), by \(h_3(2k+1,k)\), or by
\(h_2(2k+1,k,c)\) with
\[
 c\in\{k+2,\ldots,2k-1,2k+1\},
\]
unless \(\calF=\calH_1(X,M)\).  All the former possibilities contradict
\eqref{eq:critical-list}.

Now let \(\tau(\calF)=m\ge3\).  If \(m=k\),
\cite[Lemma~3.8(i)]{CaoLvWangZhou} gives
\(|\calF|\le p^k<L\).  If \(3\le m<k\), the bounds in
\cite[Lemma~3.8(ii),(iii)]{CaoLvWangZhou} are
\begin{align*}
 u_1(m)
 &=\left(\qbinom{m-1}{1}p^{m-1}
        +q^{2m-2}p^{m-2}\right)
   \qbinom{2k+1-m}{k-m},\\
 u_2(m)
 &=\qbinom{m-1}{1}\qbinom{m}{1}p^{m-2}
   \qbinom{2k+1-m}{k-m}\\
 &\quad
 +q^{m-1}\qbinom{k-m+1}{1}\qbinom{m}{1}p^{m-1}
   \qbinom{2k-m}{k-m-1}.
\end{align*}
The adjacent-ratio calculation gives \(u_2(m)<u_1(m)\).  A sufficient
inequality is
\[
 q^{k+2}>
 \frac{\qbinom{k-m+1}{1}}{\qbinom{k-m}{1}}\,
 \frac{\qbinom{m}{1}}{\qbinom{m-1}{1}}\,p,
\]
which follows from
\[
 \frac{\qbinom{r+1}{1}}{\qbinom{r}{1}}<q+1,\qquad
 p<\frac{q^k}{q-1},\qquad
 q^2>\frac{(q+1)^2}{q-1}\quad(q\ge3).
\]
Also \(u_1(m)\) decreases with \(m\).  For \(3\le m\le k-2\), it is
enough that
\[
 q^{k+1}\qbinom{m-1}{1}>
 \qbinom{m}{1}p+q^{2m}.
\]
Indeed, after division by \(q^{k+m-1}\), the left-hand side is at least
\(1+q^{-1}\), whereas the two terms on the right are smaller than
\(q/(q-1)^2\) and at most \(q^{-1}\), respectively.
Consequently,
\[
 |\calF|\le u_1(3)=U_3^*.
\]
Lemma~\ref{lem:HM-critical-comparison} now gives a contradiction unless
\((q,k)=(3,4)\).  In that exceptional case,
Lemma~\ref{lem:ore-disjoint-average} and
\eqref{eq:critical-exception-average} give
\[
 \sigma_{4,3}(\calF)
 \le\frac{p(p-1)(|\calF|-1)}{N-p}
 \le\frac{p(p-1)(U_3^*-1)}{N-p}<pD,
\]
again a contradiction.  Hence only the asserted \(\calH_1\) case remains.
\end{proof}

\begin{lemma}\label{lem:HM-boundary-comparison}
Let $q$ be a prime power and $k\ge4$.  Put
\[
 p=\qbinom{k}{1},\quad N=\qbinom{2k+2}{1},\quad
 Q=\qbinom{2k}{k-2},\quad R=\qbinom{2k-1}{k-3},
\]
\[
 D=Q-q^{k(k-2)}\qbinom{k}{2},
 \qquad L=\frac{N}{p}D.
\]
For \(3\le l\le k-1\), define
\begin{align*}
 U_0&=Q+q(q+1)\qbinom{k-1}{1}pR,\\
 U_l&=\qbinom{l-1}{1}Q
 +q^{l-1}\qbinom{k-l+1}{1}pR
 +q^k\qbinom{2k+2-l}{k-l+1},\\
 U_3^*&=\bigl((q+1)p^2+q^4p\bigr)R.
\end{align*}
Then
\begin{align}
 \max\bigl\{U_0,\ U_l,\ U_3^*,\ p^k,\ h_3(2k+2,k),\
 h_2(2k+2,k,c)\bigr\}<L                         \label{eq:boundary-list}
\end{align}
for every \(3\le l\le k-1\) and
\(c\in\{k+2,\ldots,2k-1,2k+2\}\).
\end{lemma}

\begin{proof}
The standard intersection-number formula in the Grassmannian gives
\[
 D=\sum_{j=1}^{k-2}
 q^{(k-j)(k-2-j)}
 \qbinom{k}{j}\qbinom{k}{k-2-j}.
\]
Keeping the term \(j=1\), we obtain
\[
 L\ge Nq^{(k-1)(k-3)}\qbinom{k}{3}.
\]
Moreover,
\[
 \frac{q^{(k-1)(k-3)}\qbinom{k}{3}}{R}
 =\prod_{i=4}^{k}
 \frac{1-q^{-i}}{1-q^{-(k-1+i)}}
 >\alpha_q:=1-\frac{1}{q^3(q-1)}.
\]
Consequently,
\begin{equation}\label{eq:boundary-L-lower}
 L>\alpha_qNR.
\end{equation}

We record the finite comparisons with the right-hand side of
\eqref{eq:boundary-L-lower}.  They follow by substituting the product
formula for the Gaussian coefficients and using the adjacent-ratio identity
\[
 \frac{Q}{R}=\frac{q^{2k}-1}{q^{k-2}-1}.
\]
When \(q=2\), this ratio is smaller than
\(\frac{16}{3}2^k\).  We also use
\[
 C_2:=\prod_{i=1}^{\infty}(1-2^{-i})^{-1}<\frac72.
\]
Indeed, the product of the first four factors before inversion is
\(315/1024\), and the remaining product is at least \(15/16\), so the
full product is larger than \(2/7\).  Direct simplification gives
\begin{align*}
 \frac{U_0}{R}
 &<3\cdot2^{2k}+\frac{16}{3}2^k,\\
 \max_{3\le l\le k-1}\frac{U_l}{R}
 &<\frac{10}{3}2^{2k},\\
 \frac{h_3(2k+2,k)}{R}
 &<\frac{112}{3}2^k<\frac{21}{8}2^{2k},\\
 \max_c\frac{h_2(2k+2,k,c)}{R}
 &<\frac{77}{24}2^{2k}+16.
\end{align*}
For the last line, the first two terms in the formula for \(h_2(2k+2,k,c)\) are at
most \(\qbinom{k-1}{1}Q<\frac73\,2^{2k}R\).  If
\(c\le2k-1\), write \(s=2k-c-1\), so \(0\le s\le k-3\); the third
term divided by \(R\) is then smaller than
\[
 C_2\,2^{(k-1)^2+3s-(k-3)(k+2)}
 \le\frac78\,2^{2k}.
\]
When \(c=2k+2\), that term is zero by convention.  In every case, the last
term divided by \(R\) is smaller than \(16\).
Each last expression is smaller than
\(\frac78(2^{2k+2}-1)=\alpha_2N\).  Also
\[
 \frac{U_3^*}{R}
 =3(2^k-1)^2+16(2^k-1)<\alpha_2N
 \qquad(k\ge5).
\]
For the one omitted value \((q,k)=(2,4)\), exact substitution gives
\[
 p=15,\quad R=127,\quad D=1835,\quad
 L=125147,\quad U_3^*=116205<L.
\]

For \(q\ge3\), we use
\[
 \frac{Q}{R}<\frac{q^{k+2}}{1-q^{-2}}
 \qquad\text{and}\qquad \alpha_q\ge\frac{53}{54}.
\]
The same substitutions yield
\begin{align*}
 \frac{U_0}{R}
 &<\frac38q^{2k+1},&
 \max_{3\le l\le k-1}\frac{U_l}{R}
 &<\frac{15}{8}q^{2k},\\
 \frac{h_3(2k+2,k)}{R}
 &<\frac9{16}q^{2k+1},&
 \max_c\frac{h_2(2k+2,k,c)}{R}
 &<\frac23q^{2k+1},\\
 \frac{U_3^*}{R}
 &<\frac{2q}{(q-1)^2}q^{2k}.&&
\end{align*}
For completeness, the \(h_2(2k+2,k,c)\) estimate follows from
\[
 \frac{\qbinom{k-1}{1}Q}{R}
 <\frac{q^{2k+1}}{(q-1)(1-q^{-2})}.
\]
For \(c\le2k-1\), the remaining Gaussian term, with
\(s=2k-c-1\in\{0,\ldots,k-3\}\), is smaller than
\[
 2q^{(k-1)^2+3s-(k-3)(k+2)}
 \le2q^{2k-2},
\]
because
\(\prod_{i\ge1}(1-q^{-i})^{-1}<2\) for \(q\ge3\); when
\(c=2k+2\), this term is zero.  The final term contributes less than
\(q^{2k+1}/((q-1)q^5)\).
At \(q=3\), the sum of the resulting three coefficients is already
smaller than \(2/3\), and it decreases with \(q\).
All five expressions are smaller than
\(\alpha_q q^{2k+1}<\alpha_qN\).

It remains only to include \(p^k\).  Since
\[
 R>q^{(k-3)(k+2)},\qquad N>q^{2k+1},\qquad
 p<\frac{q^k}{q-1},
\]
we have
\[
 \frac{p^k}{NR}<\frac{q^{5-k}}{(q-1)^k}.
\]
This is smaller than \(\alpha_q\) for \(q\ge3\), and also for
\(q=2,\ k\ge6\).  The remaining exact comparisons are
\[
 15^4<125147,\qquad
 31^5<167587875=L\big|_{(q,k)=(2,5)}.
\]
Together with \eqref{eq:boundary-L-lower}, these inequalities prove
\eqref{eq:boundary-list}.
\end{proof}

\begin{lemma}\label{lem:HM-boundary-structure}
Let $q$ be a prime power, $k\ge4$ and \(n=2k+2\).  With \(L\) as in
Lemma~\ref{lem:HM-boundary-comparison}, let
$\calF\subseteq\vbinom{V}{k}$ be a maximal nontrivial intersecting family.
If \(|\calF|\ge L\), then
\[
 \calF=\calH_1(X,M)
\]
for some \(X<M\le V\) with \((\dim X,\dim M)=(1,k+1)\).
\end{lemma}

\begin{proof}
The covering-number analysis in
\cite[Lemmas~3.3--3.8]{CaoLvWangZhou} is valid under \(2k\le n\), even
though the final large-family theorem in that paper is stated in a smaller
range.  We specialize those lemmas to \(t=1\).

First suppose \(\tau(\calF)=2\), and let \(\calT\) be the family of all
\(2\)-dimensional covers.  If \(|\calT|=1\),
\cite[Lemma~3.7(i)]{CaoLvWangZhou} gives \(|\calF|\le U_0\).  If its
members have a common point and span an \(l\)-subspace with
\(3\le l\le k-1\), \cite[Lemma~3.7(ii)]{CaoLvWangZhou} gives
\(|\calF|\le U_l\).  If they have no common point,
\cite[Lemmas~3.3 and 3.6]{CaoLvWangZhou} gives
\(\calF=\calH_3^{(k)}(Z)\).  The remaining common-point cases in
\cite[Lemmas~3.4 and 3.5]{CaoLvWangZhou} give either
\(\calF=\calH_1(X,M)\), or
\(\calF=\calH_2(X,M,C)\) with
\[
 \dim C\in\{k+2,\ldots,2k-1,2k+2\}.
\]
Lemma~\ref{lem:HM-boundary-comparison} excludes every one of these
possibilities except \(\calH_1(X,M)\).

Now let \(\tau(\calF)=m\ge3\).  If \(m=k\),
\cite[Lemma~3.8(i)]{CaoLvWangZhou} gives \(|\calF|\le p^k<L\).
If \(3\le m<k\), the two bounds in
\cite[Lemma~3.8(ii),(iii)]{CaoLvWangZhou} are
\begin{align*}
 u_1(m)
 &=\left(\qbinom{m-1}{1}p^{m-1}
        +q^{2m-2}p^{m-2}\right)
   \qbinom{2k+2-m}{k-m},\\
 u_2(m)
 &=\qbinom{m-1}{1}\qbinom{m}{1}p^{m-2}
   \qbinom{2k+2-m}{k-m}\\
 &\quad
 +q^{m-1}\qbinom{k-m+1}{1}\qbinom{m}{1}p^{m-1}
   \qbinom{2k+1-m}{k-m-1}.
\end{align*}
The adjacent-ratio formula shows \(u_2(m)<u_1(m)\).  Indeed, after
normalizing their difference, it is enough that
\[
 q^{k+3}>
 \frac{\qbinom{k-m+1}{1}}{\qbinom{k-m}{1}}\,
 \frac{\qbinom{m}{1}}{\qbinom{m-1}{1}}\,p,
\]
which follows from
\[
 \frac{\qbinom{r+1}{1}}{\qbinom{r}{1}}<q+1
 \quad\text{and}\quad
 q^3>\frac{(q+1)^2}{q-1}\quad(q\ge3).
\]
For \(q=2\), the two adjacent ratios are at most \(3\) and \(7/3\),
respectively, so their product is smaller than \(8=q^3\).
Moreover \(u_1(m)\) decreases with \(m\).  For \(3\le m\le k-2\), a
sufficient form of the adjacent comparison is
\[
 q^{k+2}\qbinom{m-1}{1}>
 \qbinom{m}{1}p+q^{2m}.
\]
After division by \(q^{k+m-1}\), the left-hand side is
\(q^{3-m}\qbinom{m-1}{1}\ge q+1\), whereas the two terms on the
right are smaller than \(q/(q-1)^2\) and at most \(q^{-1}\),
respectively.  Since
\[
 q+1>\frac{q}{(q-1)^2}+q^{-1}
 \qquad(q\ge2),
\]
the required comparison follows.  Therefore
\[
 |\calF|\le u_1(3)
 =\bigl((q+1)p^2+q^4p\bigr)R=U_3^*<L,
\]
again a contradiction.  Hence only the asserted \(\calH_1(X,M)\) case remains.
\end{proof}

Now we are going to prove Theorem~\ref{thm:ore-hm-sharp}.

\begin{proof}[Proof of Theorem~\ref{thm:ore-hm-sharp}]
Let $\calF$ be nontrivial and intersecting, and extend it to a maximal
intersecting family $\calF^*$.  The extension remains nontrivial.  By
Observation~\ref{obs:monotone},
\[
 \sigma(\calF)\le\sigma(\calF^*).
\]
Put
\[
 p=\qbinom{k}{1},\qquad
 Q=\qbinom{n-2}{k-2},\qquad
 D=Q-q^{k(k-2)}\qbinom{n-k-2}{k-2},
\qquad
 L=\frac{\qbinom{n}{1}}{p}D.
\]

We first record the Ore-structural consequence
\begin{equation}\label{eq:HM-envelope}
 \sigma_{k,q}(\calF^*)\ge pD
 \quad\Longrightarrow\quad
 \calF^*=\calH_1(X,M)
 \ \text{or}\
 \bigl(k=3\ \text{and}\ \calF^*=\calH_3(Z)\bigr).
\end{equation}
For $k=3$, Proposition~\ref{prop:h3-degree} gives $D=p$, so
$L=\qbinom{n}{1}=:N$.  Lemma~\ref{lem:ore-size} gives
$|\calF^*|\ge N$.  De Boeck classified the large maximal intersecting
families of $3$-subspaces; see \cite{DeBoeck} and the eleven-family list in
\cite{IhringerKupavskii}.  The full list in the notation used
here is recorded in Appendix~\ref{app:k3-classification}.  Since
\[
 \qbinom{7}{1}-
 (3q^4+3q^3+2q^2+q+1)
 =q^2(q^4+q^3-2q^2-2q-1)>0,
\]
the size hypothesis places \(\calF^*\) in that list whenever \(n\ge7\).
Type \(1\) is a point-star and is impossible because \(\calF^*\) is
nontrivial, while Types \(2\) and \(3\) are \(\calH_1\) and
\(\calH_3\).  For \(n\ge8\), every remaining type has size smaller than
\(N\): Types \(5\)--\(9\) and \(11\) are already smaller than
\(\qbinom{7}{1}\), and at \(n=8\) the two remaining comparisons are
\[
 q^2(q^5-q^3-3q^2-2q-1)>0,\qquad
 q^2(q^5-q^2-q-1)>0
\]
for Types \(4\) and \(10\), respectively; each difference increases with
\(n\).

It remains in this argument to consider \(n=7\), which is part of the
theorem only when \(q\ge3\).  Type \(4\) is excluded by
\[
 \qbinom{7}{1}-|\mathcal F_{\rm type\,4}|
 =q^2(q^4-3q^2-2q-1)>0,
\]
and Types \(5\)--\(9\) and \(11\) are again smaller than \(N\).
For Type \(10\), there is a fixed \(5\)-space \(Y\) and
\(\calF^*=\vbinom{Y}{3}\).  Every point of \(Y\) has degree
\(\qbinom{4}{2}\), every point outside \(Y\) has degree zero, and a
\(3\)-space meeting \(Y\) in exactly one point exists.  Hence
\[
 \sigma_{3,q}\!\left(\vbinom{Y}{3}\right)=\qbinom{4}{2}
 <p^2,
 \qquad
 p^2-\qbinom{4}{2}=qp>0.
\]
Thus Type \(10\) is incompatible with the hypothesis in
\eqref{eq:HM-envelope}.  This proves \eqref{eq:HM-envelope} for \(k=3\).

Let $k\ge4$.  If $n=2k+2$, the implication is
Lemma~\ref{lem:HM-boundary-structure}, after applying
Lemma~\ref{lem:ore-size} to obtain \(|\calF^*|\ge L\).
If $n\ge2k+3$, the same size implication together with
Lemma~\ref{lem:HM-separation} shows that $L$ is strictly larger than the
structural threshold and than every alternative in
Theorem~\ref{thm:CLWZ-structure} except $\calH_1$.  Finally, if
\(n=2k+1\), the hypotheses of the theorem give \(q\ge3\), and
Lemma~\ref{lem:HM-critical-structure} applies directly.  This proves
\eqref{eq:HM-envelope} in all permitted cases.

Suppose for a contradiction that $\sigma(\calF)>pD$.  By
Observation~\ref{obs:monotone},
\(\sigma(\calF^*)\ge\sigma(\calF)>pD\), so
\eqref{eq:HM-envelope} applies.  Propositions
\ref{prop:hm-degree} and \ref{prop:h3-degree}, together with
Observation~\ref{obs:monotone}, give
\[
 \sigma(\calF)\le\sigma(\calF^*)=pD,
\]
a contradiction.  This proves the upper bound, and the same two
propositions prove sharpness.

It remains to classify equality.  Suppose that $\sigma(\calF)=pD$.
Observation~\ref{obs:monotone} gives
\(\sigma(\calF^*)\ge pD\), so \eqref{eq:HM-envelope} shows that either
\[
 \calF^*=\calH_1(X,M)
\]
for a \(X<M\le V\) with $(\dim X,\dim M)=(1,k+1)$, or $k=3$ and
$\calF^*=\calH_3(Z)$.

Suppose first that $\calF\subseteq\calH_1(X,M)$.  If
$S\in\vbinom{V}{k}$ is disjoint from $M$, then $S\notin\calF$, and
Proposition~\ref{prop:hm-degree} gives
\[
 pD=\sigma(\calF)\le d_{\calF}(S)
 \le d_{\calH_1(X,M)}(S)=pD.
\]
Every point of $S$ therefore has $\calF$-degree $D$, since its degree is at
most its degree $D$ in $\calH_1(X,M)$.  Moreover, every point
$y\not\le M$ lies in such a subspace $S$: choose a $k$-subspace of $V/M$
containing the image of $y$ and lift a basis so that the lift contains $y$.
Thus
\begin{equation}\label{eq:HM-equality-outside-degree}
 d_{\calF}(y)=D\qquad(y\not\le M).
\end{equation}
If a member $E$ of $\calH_1(X,M)\setminus\calF$ is not contained in
$M$, any point $y\le E$ outside $M$ would have
$d_{\calF}(y)<d_{\calH_1(X,M)}(y)=D$, contradicting
\eqref{eq:HM-equality-outside-degree}.  Hence
\[
 \calF=\calH_1^\circ(X,M)\cup\calG
\]
for some $\calG\subseteq\vbinom{M}{k}$.

The intersection of all members of $\calH_1^\circ(X,M)$ is exactly $X$.
To see this, for any point $z\ne X$ choose a $2$-subspace
$L$ with $X<L\le M$ and $z\nleq L$ when $z\le M$, and then extend $L$ to
a $k$-subspace $E$ satisfying $E\cap M=L$ and $z\nleq E$.  Such an
extension exists because $\dim(V/M)=n-k-1\ge k$: if $z\not\le M$, choose
the $(k-2)$-dimensional image of $E$ in $V/M$ to avoid the image of $z$.
It follows that the
displayed family is nontrivial if and only if $\calG$ contains a member
which does not contain $X$.  This proves the necessity of (i).

Conversely, let $\calF=\calH_1^\circ(X,M)\cup\calG$, where $\calG$ has
the stated property.  Then $\calF$ is a nontrivial intersecting family.
Every point outside $M$ has degree exactly $D$.  If $z\le M$ and
$z\ne X$, then
\[
 d_{\calF}(z)\ge
 d_{\calH_1^\circ(X,M)}(z)
 =Q-\qbinom{k-1}{1}.
\]
The same lower bound holds for $d_{\calF}(X)$, by fixing one such point
$z$ and counting the outer members through $X+z$.  Moreover,
\[
 q^{k(k-2)}\qbinom{n-k-2}{k-2}
 \ge q^{k(k-2)}>\qbinom{k-1}{1},
\]
because \(\qbinom{k-1}{1}<q^{k-1}\) and \(k(k-2)\ge k-1\).
Hence every point of $V$ has $\calF$-degree at least $D$.
Therefore every
nonmember has degree sum at least $pD$, while a $k$-subspace disjoint from
$M$ is a nonmember whose points all have degree $D$.  Thus
$\sigma(\calF)=pD$, proving the sufficiency of (i).

It remains to consider $\calF\subseteq\calH_3(Z)$, where necessarily
$k=3$ and $p=\qbinom{3}{1}$.  For every $3$-subspace $S$ disjoint from
$Z$,
\[
 p^2=\sigma(\calF)\le d_{\calF}(S)
 \le d_{\calH_3(Z)}(S)=p^2.
\]
As above, every point outside $Z$ has $\calF$-degree $p$, and every such
point lies in a $3$-subspace disjoint from $Z$.  Hence no deleted member of
$\calH_3(Z)$ can contain a point outside $Z$.  The only member contained in
$Z$ is $Z$ itself, so
\[
 \calF=\calH_3(Z)\qquad\hbox{or}\qquad
 \calF=\calH_3(Z)\setminus\{Z\}.
\]

Conversely, deleting $Z$ does not change the degree $p$ of any point
outside $Z$.  For a point $z\le Z$, its degree in
$\calH_3(Z)\setminus\{Z\}$ is
\[
 (q+1)\left(\qbinom{n-2}{1}-1\right)\ge p:
\]
choose the $q+1$ planes of $Z$ through $z$, and then a $3$-subspace
through the chosen plane other than $Z$.  Thus both displayed families are
nontrivial: for each point one can choose a plane of $Z$ and an outer
extension of that plane which avoids the point.  All their point degrees are
at least $p$, and a $3$-subspace disjoint from $Z$ witnesses Ore-degree
$p^2$.  This proves (ii) and completes the equality classification.
\end{proof}

\section{Direct-sum matchings and the Ramsey corollary}\label{sec:matching}

We compare the Ore lower bound forced by the construction
$\calB_q(V,k;W)$ with the stability extremal number in
Theorem~\ref{thm:linear-stability}.

\begin{lemma}\label{lem:matching-separation}
Let $q$ be a prime power, $k\ge2$, $m\ge2$, and
$n\ge(2m+1)k-m+3$.  Then
\begin{equation}\label{eq:L-greater-ell}
 \frac{\qbinom{n}{1}}{\qbinom{k}{1}}D_q(n,k,m)
 >\ell_q(n,k,m).
\end{equation}
\end{lemma}

\begin{proof}
Put
\[
 Q=\qbinom{n-2}{k-2},\quad
 R=\qbinom{n-1}{k-1},\quad
 X=\frac{\qbinom{n}{1}}{\qbinom{k}{1}},\quad
 Y=\frac{R}{Q}=\frac{\qbinom{n-1}{1}}{\qbinom{k-1}{1}},
 \quad c_0=\frac{31}{32}.
\]
Since $n-m-k+2\ge2m(k-1)+5\ge9$, the lower bound in
Lemma~\ref{lem:gaussian-difference}, applied with
$(M,t,r)=(n-1,m,k-1)$, gives
\begin{equation}\label{eq:D-matching-finite-lower}
 D_q(n,k,m)\ge c_0\qbinom{m}{1}Q.
\end{equation}

Iterating \eqref{eq:ell-recurrence} gives
\begin{equation}\label{eq:ell-expanded}
 \ell_q(n,k,m)=
 \sum_{j=0}^{m-2}q^{kj}\qbinom{n-j-1}{k-1}
 +q^{k(m-1)}\ell_q(n-m+1,k,1).
\end{equation}
The definition \eqref{eq:linear-HM-base} gives, for every admissible $N$,
\[
 \ell_q(N,k,1)=
 \qbinom{N-1}{k-1}
 -q^{k(k-1)}\qbinom{N-k-1}{k-1}+q^k.
\]
For $0\le j\le m-2$, comparison of consecutive Gaussian coefficients gives
\[
 q^{kj}\qbinom{n-j-1}{k-1}\le q^jR,
\]
and hence the sum in \eqref{eq:ell-expanded} is at most
$\qbinom{m-1}{1}R$.  Applying the Gaussian difference estimate with
$(M,t,r)=(n-m,k,k-1)$ also gives
\begin{align*}
 q^{k(m-1)}\ell_q(n-m+1,k,1)
 &\le
 q^{k(m-1)}\qbinom{k}{1}\qbinom{n-m-1}{k-2}+q^{km}\\
 &\le q^{2m-2}\qbinom{k}{1}Q+q^{km},
\end{align*}
where in the last step we used
\[
 \frac{\qbinom{n-m-1}{k-2}}{Q}
 \le q^{-(m-1)(k-2)}.
\]
Therefore
\begin{equation}\label{eq:ell-finite-upper}
 \ell_q(n,k,m)
 \le\qbinom{m-1}{1}R
 +q^{2m-2}\qbinom{k}{1}Q+q^{km}.
\end{equation}

Now
\[
 \frac{X}{Y}>
 \frac{1-q^{-(k-1)}}{1-q^{-k}}
 \ge\frac{q}{q+1}.
\]
Since $\qbinom{m}{1}=q\qbinom{m-1}{1}+1$ and
\[
 \frac{31}{32}\frac{q^2}{q+1}-1\ge\frac q8
 \qquad(q\ge2),
\]
we obtain
\[
 c_0\qbinom{m}{1}\frac{X}{Y}-\qbinom{m-1}{1}
 >\frac{q^{m-1}}8.
\]
It follows that the difference between
$c_0\qbinom{m}{1}XQ$ and the first term on the right of
\eqref{eq:ell-finite-upper} is greater than
\[
 \frac{q^{n-k+m-1}}8Q\ge q^{n-k+m-4}Q.
\]
The two remaining terms in \eqref{eq:ell-finite-upper}, after division by
$Q$, are respectively smaller than
\[
 q^{2m+k-2}
 \quad\text{and at most}\quad
 q^{km-(k-2)(n-k)}.
\]
Both exponents are at most $n-k+m-7$: indeed, the two relevant differences
are
\[
 n-2k-m-2\ge3,
 \qquad (k-1)(n-k-m)-4\ge3.
\]
Their sum is therefore smaller than $q^{n-k+m-4}$.  Combining this with
\eqref{eq:D-matching-finite-lower} and \eqref{eq:ell-finite-upper} proves
\eqref{eq:L-greater-ell}.
\end{proof}

\begin{proof}[Proof of Theorem~\ref{thm:ore-matching}]
If $s=2$, the expression in parentheses in
\eqref{eq:ore-matching-bound} equals $\qbinom{n-2}{k-2}$ by the Gaussian
recurrence.  A family with no direct-sum matching of size two is
intersecting, so the conclusion follows from Theorem~\ref{thm:ore-ekr}.

Now let $s\ge3$, put $m=s-1$, and suppose that
$\nu_{\oplus}(\calF)\le m$.  If $\calF$ is
contained in some $\calB_q(V,k;W)$ with $\dim W=m$, then
Observation~\ref{obs:monotone} and Proposition~\ref{prop:B-degree} give
\[
 \sigma(\calF)\le\qbinom{k}{1}D_q(n,k,m),
\]
a contradiction with \eqref{eq:ore-matching-bound}.

If $\calF$ is not contained in such a family, then
Theorem~\ref{thm:linear-stability} gives
\[
 |\calF|\le\ell_q(n,k,m)
\]
provided $n\ge(2m+1)k-m+3$.  On the other hand,
Lemma~\ref{lem:ore-size} and \eqref{eq:ore-matching-bound} give
\[
 |\calF|>
 \frac{\qbinom{n}{1}}{\qbinom{k}{1}}D_q(n,k,m),
\]
contradicting Lemma~\ref{lem:matching-separation}.  Hence
$\nu_{\oplus}(\calF)\ge s$.

The sharpness statement follows from Proposition~\ref{prop:B-degree} with
$\dim W=s-1$.
\end{proof}

\begin{proof}[Proof of Corollary~\ref{cor:ramsey}]
If $M=0$, then every $a_i=1$.  The inequality \eqref{eq:ramsey-ore}
reduces to $\sigma(\calF)>0$, so $\calF$ is nonempty and the conclusion is
immediate.  Assume $M\ge1$ and apply Theorem~\ref{thm:ore-matching} with
$s=M+1$.  The family $\calF$ contains a direct-sum matching of size $M+1$.
If colour $i$ occurred at most $a_i-1$ times on this matching for every $i$,
then the matching would have at most
\[
 \sum_{i=1}^{c}(a_i-1)=M
\]
members, a contradiction.  Thus some colour $i$ occurs on at least $a_i$
members, which form the required monochromatic direct-sum matching.
\end{proof}

There is a second natural obstruction outside the stability range.  Let $H$
be an $(sk-1)$-subspace and put $\calF=\vbinom{H}{k}$.  Then
$\nu_{\oplus}(\calF)\le s-1$.  Points in $H$ have degree
$\qbinom{sk-2}{k-1}$ in $\calF$, whereas points outside $H$ have degree
zero.  The minimum possible dimension of $S\cap H$ over nonmembers
$S\in\vbinom{V}{k}$ is $\max\{0,(s+1)k-n-1\}$.  Consequently,
\begin{equation}\label{eq:clique-barrier-ore}
 \sigma_{k,q}(\calF)=
 \qbinom{\max\{0,(s+1)k-n-1\}}{1}
 \qbinom{sk-2}{k-1}.
\end{equation}
In particular, a conjecture using only the threshold
\eqref{eq:ore-matching-bound} is false in the full range.  For example, when
$(q,k,s,n)=(2,3,4,13)$, the value in
\eqref{eq:clique-barrier-ore} is $174251$, whereas the right-hand side of
\eqref{eq:ore-matching-bound} is $100205$.  The two standard matching
barriers suggest the following corrected formulation.

\begin{conj}
\label{conj:full-range-matching}
Let $q$ be a prime power, $k\ge2$, $s\ge2$, and $n\ge ks$.  If
$\calF\subsetneq\vbinom{V}{k}$ satisfies
\[
 \sigma_{k,q}(\calF)>
 \max\left\{\qbinom{k}{1}D_q(n,k,s-1),
 \qbinom{\max\{0,(s+1)k-n-1\}}{1}
 \qbinom{sk-2}{k-1}\right\},
\]
then
$\nu_{\oplus}(\calF)\ge s$.
\end{conj}

\section{Concluding remarks}\label{sec:conclusion}

The projective Ore-degree in \eqref{eq:ore-definition} is a robust
Grassmannian analogue of the vertex Ore-degree for uniform hypergraphs.  The
incidence inequality \eqref{eq:ore-size} is exact and requires no asymptotic
hypothesis.  Theorem~\ref{thm:ore-hm-sharp} gives the exact
Ore--Hilton--Milner threshold for \(q\ge3,n\ge2k+1\), and for
\(q\ge2,n\ge2k+2\), by combining that inequality with large-family
structure, boundary covering-number estimates, disjoint-subspace averaging,
and finite Gaussian comparisons.  Its equality classification also exhibits
a precise rigidity--flexibility dichotomy: all members using points outside
the distinguished core space are forced, whereas the internal layer can be
thinned arbitrarily subject to nontriviality.

The distinction between partial-spread matchings and direct-sum matchings is
structural rather than cosmetic.  If a matching were defined only by pairwise
trivial intersections, the family of all $k$-subspaces meeting a fixed
$(s-1)$-subspace would no longer have matching number at most $s-1$: many
pairwise disjoint points can lie in a low-dimensional projective space.  In
that setting, unions of point-stars depend on the rank function of the centre
configuration, and the set-system threshold in Theorem~\ref{thm:ore-matching}
does not survive by a mechanical replacement of binomial coefficients with
Gaussian coefficients.  The direct-sum definition restores the rank-additive
feature of disjoint sets and produces the canonical construction
\eqref{eq:B-family}.

Two natural problems remain within the scope of the present paper.  The first
is the binary critical case recorded in
Remark~\ref{rema:hm-binary-boundary}: the structural theorem of Ihringer and
Kupavskii reaches $n=2k+1$, but a sharper treatment of its remainder is still
needed when \(q=2\).  The second is to lower the dimension threshold in
Theorem~\ref{thm:ore-matching}, ultimately toward the two-barrier
Conjecture~\ref{conj:full-range-matching}.

\begin{rema}[Further extensions]
The projective incidence method suggests three broader directions that are
not pursued here.  First, one may seek sharp product bounds for the
projective Ore-degrees of cross-intersecting families.  Second, one may
replace point degrees by degrees of $t$-subspaces and develop a higher-order
$t$-Ore theory for $t$-intersecting families; a suitable nondegenerate local
parameter and its sharp extremal constructions would have to be identified.
Third, one may formulate rainbow versions for several properly coloured
families and ask for projective Ore conditions forcing a rainbow direct-sum
matching.  A sharp result in the last direction is likely to require a
rainbow stability theorem for direct-sum matchings.
\end{rema}
\section*{Acknowledgement}
The authors acknowledge the use of AI tools during the exploratory stage of this project. All mathematical arguments and proofs in the final manuscript were checked and written by the authors.

\appendix
\refstepcounter{section}
\section*{Appendix~\thesection.\ Large maximal intersecting families for \(k=3\)}
\label{app:k3-classification}

For completeness, we record the classification of large maximal intersecting
families of \(3\)-subspaces due to De Boeck \cite{DeBoeck}, in the condensed
vector-space form given in \cite{IhringerKupavskii}.  Throughout
this appendix, an unqualified \(F\) ranges over \(\vbinom{V}{3}\).

\begin{thm}[De Boeck]\label{thm:k3-classification}
Let \(n\ge6\), and let
\(\calF\subseteq\vbinom{V}{3}\) be a maximal intersecting family satisfying
\[
 |\calF|\ge 3q^4+3q^3+2q^2+q+1.
\]
Then, up to isomorphism, \(\calF\) is one of the following eleven types.
\end{thm}

\begin{enumerate}[\rm(1)]
\item
For a point \(P\in\vbinom{V}{1}\), let
\[
 \calF=\{F:P\le F\}.
\]
Then
\[
 |\calF|=\qbinom{n-1}{2}.
\]

\item
Let \(P\in\vbinom{V}{1}\) and \(Y\in\vbinom{V}{3}\) satisfy
\(P\cap Y=\{0\}\).  Let
\[
 \calF=
 \{F:P\le F,\ \dim(F\cap Y)\ge1\}
 \cup
 \{F:Y\le P+F\}.
\]
Then
\[
 |\calF|=
 \qbinom{3}{1}\qbinom{n-2}{1}
 -q\qbinom{2}{1}.
\]

\item
For \(X\in\vbinom{V}{3}\), let
\[
 \calF=\{F:\dim(F\cap X)\ge2\}.
\]
Then
\[
 |\calF|=
 \qbinom{3}{1}\qbinom{n-2}{1}
 -q\qbinom{2}{1}.
\]

\item
Let \(P\in\vbinom{V}{1}\), \(X\in\vbinom{V}{3}\), and
\(Y\in\vbinom{V}{5}\) satisfy \(P\le X\le Y\).  Let
\[
 \begin{split}
 \calF={}&
 \{F:P\le F\le Y\}\\
 &{}\cup\{F:P\le F,\ \dim(F\cap X)=2\}\\
 &{}\cup\{F:F\le Y,\ \dim(F\cap X)=2\}.
 \end{split}
\]
Then
\[
 |\calF|=
 \qbinom{2}{1}\qbinom{n-2}{1}
 +2q^4+q^3-q.
\]

\item
Let \(L\in\vbinom{V}{2}\) and \(Y\in\vbinom{V}{5}\) satisfy \(L\le Y\).
Let
\[
 \calF=
 \{F:L\le F\}
 \cup
 \{F:F\le Y,\ \dim(F\cap L)=1\}.
\]
Then
\[
 |\calF|=
 \qbinom{n-2}{1}
 +q^2\qbinom{2}{1}\qbinom{3}{1}.
\]

\item
Let \(P_1,P_2\in\vbinom{V}{1}\) be distinct, \(L\in\vbinom{V}{2}\),
\(X_1,X_2\in\vbinom{V}{3}\), \(Y\in\vbinom{V}{4}\), and
\(Z_1,Z_2\in\vbinom{V}{5}\), where
\[
 P_1,P_2\le L=X_1\cap X_2,
 \qquad
 X_1,X_2\le Y=Z_1\cap Z_2.
\]
Let
\[
 \begin{split}
 \calF={}&
 \{F:L\le F\}\cup\{F:F\le Y\}\\
 &{}\cup
 \bigcup_{i=1}^{2}
 \{F:P_i\le F\le Z_1,\ \dim(F\cap X_i)\ge2\}\\
 &{}\cup
 \bigcup_{\{i,j\}=\{1,2\}}
 \{F:P_i\le F\le Z_2,\ \dim(F\cap X_j)\ge2\}.
 \end{split}
\]
Then
\[
 |\calF|=\qbinom{n-2}{1}+5q^3+q^2.
\]

\item
Let \(L\in\vbinom{V}{2}\), \(X\in\vbinom{V}{4}\), and
\(Y\in\vbinom{V}{6}\) satisfy \(Y=L\oplus X\).  Choose distinct points
\(P_1,P_2,P_3\le L\) and points \(Q_1,Q_2,Q_3,Q_4\le X\) satisfying
\(X=Q_1\oplus Q_2\oplus Q_3\oplus Q_4\), and put
\[
 \begin{array}{lll}
 L_1=Q_1+Q_2,&L_2=Q_1+Q_3,&L_3=Q_1+Q_4,\\
 \overline L_1=Q_3+Q_4,&
 \overline L_2=Q_2+Q_4,&
 \overline L_3=Q_2+Q_3.
 \end{array}
\]
Let
\[
 \begin{split}
 \calF={}&
 \{F:L\le F\}\\
 &{}\cup
 \bigcup_{i=1}^{3}\{F:P_i\le F\le L+L_i\}\\
 &{}\cup
 \bigcup_{i=1}^{3}\{F:P_i\le F\le L+\overline L_i\}.
 \end{split}
\]
Then
\[
 |\calF|=\qbinom{n-2}{1}+6q^2.
\]

\item
Let \(P_1,P_2\in\vbinom{V}{1}\) be distinct, \(L\in\vbinom{V}{2}\), and
\(X_1,X_2\in\vbinom{V}{4}\), where
\[
 P_1,P_2\le L=X_1\cap X_2.
\]
Let
\[
 \begin{split}
 \calF={}&
 \{F:L\le F\}\\
 &{}\cup
 \{F:P_1\le F,\ \dim(F\cap X_1)\ge2,\
                     \dim(F\cap X_2)\ge2\}\\
 &{}\cup\{F:P_2\le F\le X_1\}
 \cup\{F:P_2\le F\le X_2\}.
 \end{split}
\]
Then
\[
 |\calF|=\qbinom{n-2}{1}+q^4+2q^3+3q^2.
\]

\item
Let \(P_1,P_2\in\vbinom{V}{1}\) be distinct, \(L\in\vbinom{V}{2}\), and
\(X\in\vbinom{V}{4}\), where \(P_1,P_2\le L\) and \(L\cap X=\{0\}\).
Let
\[
 R_1,\ldots,R_{q+1},R'_1,\ldots,R'_{q+1}\in\vbinom{X}{2}
\]
satisfy
\[
 \dim(R_i\cap R'_j)=1\quad\text{for all }i,j,
\]
and
\[
 R_i\cap R_j=R'_i\cap R'_j=\{0\}
 \quad\text{whenever }i\ne j.
\]
Thus \(\{R_1,\ldots,R_{q+1}\}\) is a regulus and
\(\{R'_1,\ldots,R'_{q+1}\}\) is its opposite regulus.  Let
\[
 \begin{split}
 \calF={}&
 \{F:L\le F\}\\
 &{}\cup
 \bigcup_{i=1}^{q+1}\{F:P_1\le F\le L+R_i\}\\
 &{}\cup
 \bigcup_{i=1}^{q+1}\{F:P_2\le F\le L+R'_i\}.
 \end{split}
\]
Then
\[
 |\calF|=\qbinom{n-2}{1}+2q^2\qbinom{2}{1}.
\]

\item
For \(X\in\vbinom{V}{5}\), let
\[
 \calF=\{F:F\le X\}.
\]
Then
\[
 |\calF|=\qbinom{5}{3}.
\]

\item
Let \(L_1,L_2\in\vbinom{V}{2}\), \(X\in\vbinom{V}{4}\), and
\(Y_1,Y_2\in\vbinom{V}{5}\), where
\[
 L_1,L_2\le X=Y_1\cap Y_2,
 \qquad
 L_1\cap L_2=\{0\}.
\]
Let
\[
 \begin{split}
 \calF={}&
 \{F:F\le X\}
 \cup\{F:L_1\le F\le Y_1\}
 \cup\{F:L_2\le F\le Y_1\}\\
 &{}\cup
 \{F:F\le Y_2,\ \dim(F\cap L_1)\ge1,\
                    \dim(F\cap L_2)\ge1\}.
 \end{split}
\]
Then
\[
 |\calF|=q^4+3q^3+4q^2+q+1.
\]
\end{enumerate}

In the notation of the main text, Types \(1\), \(2\), and \(3\) are,
respectively, a full point-star, \(\calH_1(P,P+Y)\), and \(\calH_3(X)\).

\begin{thebibliography}{99}

\bibitem{BaloghPalmerRaeisi}
J. Balogh, C. Palmer and G. Raeisi,
\newblock Matchings in hypergraphs via Ore-degree conditions,
\newblock arXiv:2603.06415, 2026.

\bibitem{BlokhuisEtAl}
A. Blokhuis, A. E. Brouwer, A. Chowdhury, P. Frankl, T. Mussche,
B. Patk\'os and T. Sz\H{o}nyi,
\newblock A Hilton--Milner theorem for vector spaces,
\newblock {\em Electron. J. Combin.} 17 (2010), R71.

\bibitem{CaoLvWangZhou}
M. Cao, B. Lv, K. Wang and S. Zhou,
\newblock Nontrivial $t$-intersecting families for vector spaces,
\newblock {\em SIAM J. Discrete Math.} 36 (2022), 1823--1847.

\bibitem{DeBoeck}
M. De Boeck,
\newblock The largest Erd\H{o}s--Ko--Rado sets of planes in finite
projective and finite classical polar spaces,
\newblock {\em Des. Codes Cryptogr.} 72 (2014), 77--117.

\bibitem{ErdosKoRado}
P. Erd\H{o}s, C. Ko and R. Rado,
\newblock Intersection theorems for systems of finite sets,
\newblock {\em Quart. J. Math. Oxford Ser. (2)} 12 (1961), 313--320.

\bibitem{FengShangguanYangZhang}
B. Feng, C. Shangguan, Y. Yang and C. Zhang,
\newblock An Erd\H{o}s matching conjecture for vector spaces,
\newblock arXiv:2606.24529, 2026.

\bibitem{FranklWilson}
P. Frankl and R. M. Wilson,
\newblock The Erd\H{o}s--Ko--Rado theorem for vector spaces,
\newblock {\em J. Combin. Theory Ser. A} 43 (1986), 228--236.

\bibitem{FranklHanHaoYi}
P. Frankl, J. Han, H. Hao and Z. Yi,
\newblock A degree version of the Hilton--Milner theorem,
\newblock {\em J. Combin. Theory Ser. A} 155 (2018), 493--502.

\bibitem{GodsilMeagher}
C. Godsil and K. Meagher,
\newblock {\em Erd\H{o}s--Ko--Rado Theorems: Algebraic Approaches},
\newblock Cambridge Studies in Advanced Mathematics 149,
Cambridge University Press, Cambridge, 2016.

\bibitem{Hsieh}
W. N. Hsieh,
\newblock Intersection theorems for systems of finite vector spaces,
\newblock {\em Discrete Math.} 12 (1975), 1--16.

\bibitem{HiltonMilner}
A. J. W. Hilton and E. C. Milner,
\newblock Some intersection theorems for systems of finite sets,
\newblock {\em Quart. J. Math. Oxford Ser. (2)} 18 (1967), 369--384.

\bibitem{HuangZhao}
H. Huang and Y. Zhao,
\newblock Degree versions of the Erd\H{o}s--Ko--Rado theorem and
Erd\H{o}s hypergraph matching,
\newblock {\em J. Combin. Theory Ser. A} 150 (2017), 233--247.

\bibitem{IhringerKupavskii}
F. Ihringer and A. Kupavskii,
\newblock Structure of $t$-intersecting families of vector spaces,
\newblock arXiv:2605.02698, 2026.

\bibitem{Ore}
O. Ore,
\newblock Note on Hamilton circuits,
\newblock {\em Amer. Math. Monthly} 67 (1960), 55.

\bibitem{ShanZhouDegree}
Y. Shan and J. Zhou,
\newblock $d$-degree Erd\H{o}s--Ko--Rado theorem for finite vector spaces,
\newblock {\em Acta Math. Hungar.} 176 (2025), 215--235.

\bibitem{WangXuZhang}
J. Wang, A. Xu and H. Zhang,
\newblock A Kruskal--Katona-type theorem for graphs: $q$-Kneser graphs,
\newblock {\em J. Combin. Theory Ser. A} 198 (2023), 105766.

\end{thebibliography}
\end{document}